\documentclass[12pt]{amsart}
\usepackage{latexsym}
\usepackage{hyperref} %
\usepackage{amsmath,amsfonts,amssymb, amsthm, amscd}
\numberwithin {equation} {section}
\newtheorem{theo}{Theorem}[section]
\newtheorem{lemma}[theo]{Lemma}
\newtheorem{corol}[theo]{Corollary}
\newtheorem{prop}[theo]{Proposition}
\theoremstyle{definition}
\newtheorem{remark}[theo]{Remark}

\newtheorem{example}[theo]{Example}

\newtheorem{defi}[theo]{Definition}
\newtheorem{note}[theo]{Note}

\newcommand{\Real}{\mathbb R}

\newcommand{\Kapa}{\mathbb K}
\newcommand{\Nat}{\mathbb N}

\newcommand{\norm}{\|\cdot\|}
\newcommand{\Lra}{\Longrightarrow}
\newcommand{\Ra}{\Rightarrow}

\newcommand{\sms}{\smallsetminus}
\newcommand{\mbx}{\mbox}
\newcommand{\rond}{\mathcal}

\newcommand{\inter}{\operatorname{int}}

\newcommand{\epi}{\operatorname{epi}} %

\newcommand{\bequ}{\begin{equation}} %
\newcommand{\eequ}{\end{equation}} %
\newcommand{\bequs}{\begin{equation*}} %
\newcommand{\eequs}{\end{equation*}} %
\newcommand{\vphi}{\varphi}
\newcommand{\epsic}{\varepsilon}
\newcommand{\W}{\Omega}
\newcommand{\Om}{\Omega}
\newcommand{\lra}{\longrightarrow}

\newcommand{\ba}{\begin{array}}
\newcommand{\ea}{\end{array}}
\newcommand{\ben}{\begin{enumerate}}
\newcommand{\een}{\end{enumerate}}
\newcommand{\ei}{\end{itemize}}
\newcommand{\bi}{\begin{itemize}}
\newcommand{\bc}{\begin{center}}
\newcommand{\ec}{\end{center}}
\newcommand{\bfr}{\begin{flushright}}
\newcommand{\efr}{\end{flushright}}
\newcommand{\f}{\frac}
\newcommand{\ov}{\overline}

\newcommand{\wtil}{\widetilde}

\newcommand{\lbda}{\lambda}

\begin{document}
\setcounter{page}{1}

\title{Lipschitz properties of  convex mappings}
\author{S. Cobza\c s}
\date{September   08, 2016}
\address{\it Babe\c s-Bolyai University, Faculty of Mathematics
and Computer Science, 400 084 Cluj-Napoca, Romania}\;\;
\email{scobzas@math.ubbcluj.ro}
\date{October 03, 2016}
\subjclass[2010]{Primary: 46N10; Secondary:  26A16, 26A51, 46A08, 46A16, 46A40,  46B40}
\keywords{convex function, convex operator, Lipschitz property,  ordered locally convex  space, cone, normal cone, normed lattice,
barrelled space, metrizale locally convex space, metric linear space}

\begin{abstract}
The present paper is  concerned with Lipschitz properties of convex mappings. One considers the general context of mappings defined on an open convex subset $\Omega$ of a locally convex space $X$  and taking values in a locally convex space $Y$  ordered by a normal cone.
One proves also equi-Lipschitz properties for   pointwise bounded families of continuous convex
mappings, provided the source space $X$ is barrelled.
 Some results  on Lipschitz properties  of continuous convex functions defined on metrizable topological vector spaces are included as well.

 The  paper has a methodological character - its aim is to show that some geometric properties (monotonicity of the slope, the normality of  the seminorms) allow to extend the proofs from the scalar case to the vector one. In this way the proofs become more transparent and natural.

\end{abstract}
\maketitle
\section{Introduction}
As it is well known  every convex function defined on an open
interval of the real axis is Lipschitz on each compact subinterval of its
domain of definition (see, e.g.,  \cite{Hardy-Lit-Pol},  Ch.3, \S18). This result can be extended to   convex functions defined
on convex open subsets of $\mathbb{R}^{n}$ - every such function is locally Lipschitz on $\W$ and Lipschitz on  every compact subset of $\W$.
Assuming the continuity of the convex function the result can be further extended to the case when $\W$ is an open convex subset of a normed space (see, e.g.,  \cite{Ekl-Tem74}), or of a locally convex space,  \cite{cob79}, \cite{co-mun76},  \cite{ma-mu96}, \cite{zali78} (see also \cite{Zali-CvAn}).

Convex mappings (convex operators, convex vector-functions), meaning mappings defined on a convex
subset of a vector space and with values in an ordered vector space, have
been intensively studied in the last years, mainly in connection with
optimization problems and mathematical programming in ordered vector spaces,
 see \cite{bor81}, \cite{bor82}, \cite{bor84}, \cite{nem86} and the monographs  \cite{Zali-Gopf}, \cite{KK95}. The normality of the cone is essential in the proofs of the continuity properties of convex vector-functions and, as it was remarked by Carioli and Vesel\'y \cite{vesely13}, the normality is, in some sense, also necessary for the validity of these properties (see Section \ref{S.cv-vect-fcs}).

Lipschitz properties of continuous convex vector functions defined on an open convex subset of a normed space and with values in a normed space ordered by a normal cone were proved in \cite{zali16} and \cite{papag83a}.

Equicontinuity results (Banach-Steinhaus type principles) for pointwise
bounded families of continuous convex mappings were proved in
\cite{k-s-w79}, \cite{neu85}. P. Kosmol \cite{Kos91} proved that a pointwise
bounded family of continuous convex mappings, defined on an open convex subset
$\Omega$ of a Banach space $X$ and with values in a normed space $Y$ ordered
by a normal cone, is locally equi-Lipschitz  on $\Omega.$ The case of
real-valued functions was considered in \cite{kos76}. M. Jouak and
L. Thibault \cite{jou-thi84} proved equicontinuity and equi-Lipschitz results
for families of continuous convex mappings defined on open convex subsets
of Baire topological vector spaces or of barrelled locally convex spaces and
taking values in a topological vector space respectively in a locally convex
space, ordered by a normal cone. New proofs of these results were given in \cite{cobz01a}. W. W.  Breckner and T. Trif \cite{br-tr99} extended  these results to
families of rationally $s$-convex functions. Condensation of singularities
principles for non-equicontinuous families of continuous convex mappings
have been proved in \cite{b-g-t96}.

 The present paper has a methodological character - its aim is to show that some geometric properties (monotonicity of the slope, the normality of  the seminorms) allow to extend the proofs from the scalar case to the vector one. In this way the proofs become more transparent and natural.

 \section{Ordered vector spaces and normal cones}\label{S.OVS}

 As we shall work with functions taking values in ordered vector spaces, we recall some notions and results on this topic. Details can be found in
 \cite{Alipr-Cones}, \cite{Hitchhick}, \cite{Breck} or \cite{Schaef-TVS}.

 A \emph{preorder}  on a nonempty set $S$ is a binary
relation on S, denoted $\leq $, which satisfies the following properties:

(O1) \;\;$s\leq s$, for all $s\in S$;

(O2)\;\; if $s\leq s^{^{\prime }}$ and $s^{^{\prime }}\leq s^{^{\prime \prime }}$, then $s\leq s^{^{^{\prime \prime }}}$,

The relation $\le $ is called an order if further

(O3)\;\;  $s\leq s^{^{\prime }}$ and $s^{^{\prime }}\leq s$ imply $s=s^{^{\prime }}$.

Two elements of $s,s'\in S$ are called \emph{comparable} if $s\le s'$ or $s'\le s$. If none of these relations hold, then the elements $s,s'\in S$ are called \emph{incomparable}. If any two elements  $s,s'\in S$ are comparable, then the set $S$ is called \emph{totally preordered} (resp. \emph{totally ordered}).

A \emph{cone} in a vector space $X$  is a nonempty subset $C$ of $X$ such that

$${\rm (C1)}\;\; C+C\subset C\quad \mbx{and}\quad {\rm (C2)}\;\; \Real_+C\subset C\,.$$

It is clear that a cone $C$ is a convex set and
$$
\alpha x+\beta y\in C\,,$$
for all $x,y\in C,$ and all $\alpha,\beta\ge 0$ in $\Real$.

The relation
\bequs
x\le_Cy\iff y-x\in C\,,
\eequs
is a \emph{vector  preorder} on $X$, i.e. a preorder satisfying  the conditions:

(OVS1)\;\; $x\le y\;\Ra x+z\le y+z$;

(OVS2) \;\; $x\le y\;\Ra tx\le ty$\,,\\
for all $x,y,z\in X$ and all $t\ge 0$.

Conversely,  if $X$ is a  vector space  equipped with a  preorder satisfying (OVS1) and (OVS2), then
$$
X_+:=\{x\in X : x\ge 0\}$$
is a cone in $X$, called the cone of positive elements, and the preorder $\le_{X_+}$ induced by $X_+$ agrees with $\le$.
A  vector preorder $\le_C$ induced by a cone $C$ is an   order if and only if the cone $C$ is \emph{pointed}, i.e.
$$
{\rm (C3)}\qquad C\cap(-C)=\{0\}\,.$$

\begin{remark}
  Some authors (see, e.g., \cite{Peressini})  use the term wedge to designate a nonempty set satisfying (C1) and (C2), and reserve the term cone for nonempty sets  satisfying (C1)--(C3).
\end{remark}

An \emph{order interval} in an ordered vector space $(X,C)$ is a (possibly empty) set of the form
\bequ\label{o-interv}
[x,y]_o=\{z\in X : x\le z\le y\}=(x+C)\cap(y-C),\eequ
for $x,y\in X.$  It is clear that an order interval $[x,y]_o$ is a convex subset of $X$ and that
$$
[x,y]_o=x+[0,y-x]_o. $$

The notation $[x,y]$ will be reserved for algebraic intervals: $$[x,y]:=\{(1-t) x+t y : t\in [0,1]\}.$$

If the elements $x,y$ are not comparable, then $[x,y]_o=\emptyset$.  If $x\le y,$ then $[x,y]\subset [x,y]_o,$ but the reverse inclusion could not hold as the following example shows.
    Taking $X=\mathbb{R}^2$ with the coordinate order and $x=(0,0),\, y=(1,1),$ then $[x,y]_o$ equals the (full) square with the vertices  $(0,0), \,(0,1),\, (1,1)$ and $(0,1),$ so it is larger  than the segment $[x,y].$

A subset $A$ of $X$ is called  \emph{full} (or \emph{order-convex}, or \emph{saturated}) if $[x,y]_o\subset A$ for all $x,y\in A.$ Since the intersection of an arbitrary  family of order--convex sets is order--convex, we can define the order--convex hull $[A]$ of a nonempty subset $A$ of $X$ as the intersection of all order-convex subsets of $X$ containing $A$, i.e. the smallest order--convex subset of $X$ containing $A$. It follows that
\begin{equation}\label{eq.o-cv-hull}
[A]=\bigcup\{[x,y]_o : x,y\in A\} =(A+C)\cap (A-C).
\end{equation}

Obviously, $A$ is order-convex iff $A=[A].$

An ordered vector space $X$ is called a \emph{vector lattice} (or a \emph{Riesz space}) if every pair $x,y\in X$ admits a supremum $x\vee y$. Since
$$
x\wedge y=-[(-x)\vee (-y)]\,,$$
it follows that every pair of elements in $X$ admits an infimum. The property extends to finite subsets of $X$, i.e. every such subset has an infimum and a supremum.

For $x\in X$ one defines
$$
x^+=x\vee 0,\quad x^-=(-x)\vee 0,\quad |x|=x\vee(-x)\,.$$

It follows
\bequ\label{eq1.v-latt}
\begin{aligned}
  & {\rm (i)}\;\; x=x^+-x^-\;\;\mbx{ and }\;\; x^+\wedge x^-=0,\quad |x|=x^++x^-,\quad |-x|=|x|; \\
  & {\rm (ii)}\;\; ||x|-|y||\le |x+y|\le |x|+|y|\,;\\
  &  {\rm (iii)}\;\; |x|\le a\iff (x\le a\;\mbx{and}\; -x\le a)  \;\mbx{ for any}\; a\ge 0\,;\\
  & {\rm (iv)}\;\; |x|\vee |y|=\frac12\big[|x+y|+|x-y|\big]\quad\mbx{and}\quad |x|\wedge |y|=\frac12\big||x+y|-|x-y|\big|\,;\\
  & {\rm (v)}\;\; x\le y\le z\;\Ra\; |y|\le |x|\vee |z|\,.
\end{aligned}\eequ

We prove only the last assertion (v) from above which will be used in the proof of Theorem \ref{t.o-Lip} (see also Remark \ref{re.o-Lip}).
The others can be found in every book on ordered vector spaces (see, for instance, \cite[Th. 1.17]{Alipr-Cones} or \cite[p. 318]{Hitchhick}).

Observe that
$$
x\le y\le z\;\Ra\; 0\le y-x\le z-x\,.$$

By (iv),
\begin{align*}
|x|\vee |z| &= \frac12\big[|z+x|+|z-x|\big]=\frac12\big[|z+x|+z-x\big]\\
&\ge \frac12\big[z+x+y-x\big]=\frac12\big[z+y\big]\ge y\,.
\end{align*}

Since
$$
x\le y\le z\;\Ra\; -z\le -y\le -x\,,$$
it follows
$$
|x|\vee |z|=|-x|\vee |-z|\ge -y\,,$$
implying  $|y|\le|x|\vee |z|$.

In fact, the following general principles hold in vector lattices (\cite[Th. 8.6 and Corollary 8.7, p. 318]{Hitchhick}).
\begin{theo}\label{t1.vect-latt} \hfill
\ben\item[\rm 1.] Every lattice identity that is true for real numbers is also true in
every Archimedean Riesz space.
\item[\rm 2.] If a lattice inequality is true for real numbers, then it is true in
any Riesz space.
\een\end{theo}

This is due to the fact that every Archimedean Riesz space is lattice isomorphic  to an appropriate function space with the order defined pointwise.

By a lattice equality (inequality) in $\Real$ one understand an equality (inequality) expressed in terms of  the order, the order operations $\sup, \inf$ and the algebraic operations with real numbers.

In the case of an ordered topological vector space (TVS) $(X,\tau)$ some connections between order and topology hold. Let  $(X,\tau)$ be a TVS with a preorder, or an  order, $\,\le\,$ generated by a cone  $C.$

We start by a simple result.
\begin{prop}\label{p1.order-tvs}
The cone   $C$ is closed if and only if the   inequalities are preserved by limits, meaning that  for all nets $(x_i : i\in I),\, (y_i : i\in I) $ in $X,$
$$
(\forall\,  i\in I,\; x_i\le y_i \; \;\;\mbox{and}\;\; \lim_ix_i=x,\, \lim_iy_i=y)\; \Longrightarrow \; x\le y.$$
\end{prop}

Other results are contained in the following proposition.

\begin{prop}[\cite{Alipr-Cones}, Lemmas 2.3 and 2.4]\label{p2.order-tvs}
  Let $(X,\tau)$  be a TVS ordered by a $\tau$-closed cone $C$. Then
  \begin{enumerate}
\item[{\rm 1.}]   The topology $\tau$ is Hausdorff.
\item[{\rm 2.}]  The cone $K$ is Archimedean.
\item[{\rm 3.}]  The order intervals are $\tau$-closed.
\item[{\rm 4.}]  If $(x_i:i\in I)$ is an increasing net which is $\tau$-convergent to $x\in X$, then $x=\sup_ix_i.$
\item[{\rm 5.}] Conversely, if the topology $\tau$ is Hausdorff, $\;\inter(K)\ne\emptyset$ and $K$ is Archimedean, then $K$ is $\tau$-closed.
\end{enumerate}
\end{prop}

\begin{note} In what follows a cone in a TVS will be always supposed to be closed.
  \end{note}

Let $(X,\tau)$ be a locally convex space ordered by a closed  cone $C$.

The cone $C$ is called {\it normal} if the space $X$ admits
a neighborhood basis at the origin   formed of  $C$-full sets. It can be
 shown that in this case $Y$ admits a basis of 0-neighborhoods formed of
 absolutely convex $C$-full sets (see \cite[V.3.1]{Schaef-TVS}).

 A seminorm $p$ on a vector space $X$ is called:
\begin{itemize}

\item \; $\gamma$-\emph{monotone} if $0\le x\le y\;\Longrightarrow\; p(x)\le \gamma p(y);$
\item \; $\gamma$-\emph{absolutely monotone} if $-y\le x\le y\;\Longrightarrow\; p(x)\le \gamma p(y);$
\item \; $\gamma$-\emph{normal} if $x\le z\le y\;\Longrightarrow\; p(z)\le \gamma\max\{p(x),p(y)\}.$  \end{itemize}

 The following characterizations of normal cones hold.
 \begin{theo}[\cite{Breck},  \cite{Schaef-TVS}]\label{t2.char-normal-cone}
 Let $(X,\tau)$ be a  LCS ordered by a cone $C.$ The following are equivalent.
   \begin{enumerate}
\item[{\rm 1.}] The cone $C$ is normal.
\item[\rm 2.] The LCS $X$ admits  a basis of 0-neighborhoods formed of  $C$-full  absolutely convex sets.
\item[{\rm 3.}] There exist $\gamma>0$ and a family   of $\gamma$-normal seminorms generating the topology $\tau$ of $X$.
\item[{\rm 4.}] There exist $\gamma>0$ and a family   of $\gamma$-monotone seminorms generating the topology $\tau$ of $X$.
\item[{\rm 5.}] There exist $\gamma>0$ and a family   of $\gamma$-absolutely monotone  seminorms generating the topology $\tau$ of $X$.
    \end{enumerate}

 All the above equivalences hold also with $\gamma =1$ in all places.
 \end{theo}

A subset $Z$ of a topological vector space $(X,\tau)$ is called \emph{bounded } (or \emph{topologically bounded}) if it is absorbed by every neighborhood of 0, i.e. for every neighborhood $V$ of 0, there exists $\lbda >0$ such that $\lbda Z\subset V$.

If $X$ is a locally convex space with the topology generated by a family $P$ of seminorms, then $Z\subset X$ is topologically bounded if and only if
$$
\sup\{p(z) : z\in Z\}<\infty\,,$$
for every $p\in P$. If, further, $X$ is a normed space, then $Z$ is topologically bounded if and only if
$$
\sup\{\|z\| : z\in Z\}<\infty\,.$$

A subset $Z$ of a vector space $(X,\le)$ ordered by a cone $C$ is called \emph{upper (lower) $o$-bounded} ($o$ comes from ``order") if there exists $y\in X$ such that $z\le y$ (resp. $ y\le z$) for all $z\in Z$, where $\le=\le_C$ is the order generated by the cone $C$. It is called $o$-\emph{bounded} if it is both upper and lower bounded, i.e. there exist $x,y\in X$ such that $Z\subset [x,y]_o$, where  $[x,y]_o$ denotes the order interval determined by $x$ and $y$ (see \eqref{o-interv}).

We mention the following result.
\begin{prop}\label{p.o-bd}
  Let $(X,\tau)$ be a topological vector space ordered by a cone $C$.
  \ben
  \item[\rm 1.] If the cone $C$ is normal, then every $o$-bounded subset of $X$ is topologically bounded.
   \item[\rm 2.] If $X$ is a Banach  space ordered by a closed  cone $C$  such that every order interval in $X$ is topologically bounded, then the cone $C$ is normal.
 \een \end{prop}
  \begin{proof}
     1.
    Suppose that the cone $C$ is normal and let $Z$ be an $o$-bounded subset of $X$. Then there exist $x,y\in  X$ such that $Z\subset [x,y]_o$. Let $V$ be a $C$-full neighborhood of $0\in X$. Since $V$ is absorbing, there exists $\lbda >0$ such that $\lbda x,\lbda y\in V$. It follows $[\lbda x,\lbda y]_o\subset[V]=V$, so that $\lbda Z\subset [\lbda x,\lbda y]_o \subset  V\,.$

    A proof of 2 is given in Step 1 of the proof of Theorem \ref{t.Vesely}.
  \end{proof}

\section{Some properties of convex vector-functions}

 We  consider now convex mappings  from a more general point of view, meaning mappings with values in an ordered vector space which are convex with respect to the  vector order and give some simple results that are  essential for the proofs in the following sections.

 Let $\,X,Y\,$ be real vector spaces and suppose that $Y$ is ordered by a
  cone $C$. If $\Omega$ is a convex subset of $X$, then a mapping
 $\,f:\Omega \to Y\,$ is called {\it convex} (or a {\it convex operator}, or $C$-\emph{convex})
 provided

 \begin{equation}\label{e1}
 f((1-\alpha)x_1+\alpha x_2) \leq
 (1-\alpha) f(x_1) + \alpha f(x_2)
 \end{equation}
 for all $\,x_1, x_2 \in \Omega\,$ and  $\,\alpha \in [0,1],\,$
 where $\, \leq := \leq _C\,$ stands for
 the order induced by the cone $C,\;\,  x\le_Cy\iff y-x\in C$.

The following results are well known in the case of real-valued convex functions.

\begin{prop}\label{p.cv-fcs2} Let $I$ be an interval in $\Real$, $Y$ a vector space  ordered by a
  cone $C$ and  $\vphi:I\to Y$ a $C$-convex function.
\ben\item[{\rm  1.}]  The following equivalent inequalities hold:
\bequ\label{eq2.cv-fcs2}
\begin{aligned}
&{\rm (a)}\quad \vphi(t_2)\leq\frac{t_3-t_2}{t_3-t_1}\,\vphi(t_1)\,+\, \frac{t_2-t_1}{t_3-t_1}\,\vphi(t_3)\,,\\
&{\rm (b)}\quad \frac{\vphi(t_2)-\vphi(t_1)}{t_2-t_1}\leq \frac{\vphi(t_3)-\vphi(t_1)}{t_3-t_1}\,,  \\
&{\rm (c)}\quad \frac{\vphi(t_3)-\vphi(t_1)}{t_3-t_1}\leq \frac{\vphi(t_3)-\vphi(t_2)}{t_3-t_2}\,,  \\
&{\rm (d)}\quad \frac{\vphi(t_2)-\vphi(t_1)}{t_2-t_1}\leq \frac{\vphi(t_3)-\vphi(t_2)}{t_3-t_2}\,,
\end{aligned}
\eequ
where $\, \leq := \leq _C\,$ is the order induced by the cone $C$.

\item[\rm 2.]
For $t_0\in I$ fixed, the slope of $\vphi$ at $t_0$, defined by
$$
\Delta_{t_0}(\vphi)(t)=\frac{\vphi(t)-\vphi(t_0)}{t-t_0},\quad t\in I\sms \{t_0\}\,,
$$
is an increasing function of $t$, i.e.
\bequ\label{eq1.slope}
\frac{\vphi(t)-\vphi(t_0)}{t-t_0}\le\frac{\vphi(t')-\vphi(t_0)}{t-t_0}\,,\eequ
for all $t,t'\in I\sms\{t_0\}$ with $t<t'$.
\een
\end{prop}\begin{proof} The proof is based on the  identity
\bequ\label{eq1.cv-fcs2}
    t_2=\frac{t_3-t_2}{t_3-t_1}t_1\,+\, \frac{t_2-t_1}{t_3-t_1}t_3,
\eequ
valid for all  points  $t_1<t_2<t_3$ in $I.$
The identity can be verified by a direct calculation.

 The inequality \eqref{eq2.cv-fcs2}.(a) follows from \eqref{eq1.cv-fcs2} and the convexity of $\vphi$.

Isolating in the left-hand side of the inequalities (b),(c),(d) the value $\vphi(t_2)$ one obtains in all cases the inequality from (a), proving their equivalence.

2 Follows from 1.
\end{proof}

For $x,y\in X,\, x\ne y,$ the right line $D(x,y)$ and the algebraic segment determined  $x,y$ are given by
 $$
 D(x,y)=\{x+t(y-x) : t\in\Real\}\quad\mbox{and}\quad [x,y]=\{x+t(y-x) : t\in[0,1]\}\,,$$
 respectively.

Consider now a more general framework.

\begin{prop}\label{cs-le.cv-Lip-lcs}
Let $X$ be a vector space and $p$ a seminorm on $X$.
For   $x,y\in X$    such that  $p(x-y)>0$    put $z_t=x+t(y-x),\, t\in\Real$.
\ben
\item[\rm 1.] For every $t,t'\in\Real$
\bequs
p(z_t-z_{t'})=|t-t'|\,p(y-x)\,.\eequs
\item[\rm 2.] If $z_1,z_2,z_3$ are points corresponding to $t_1<t_2<t_3$, then
\bequs
z_2=\frac{p(z_3-z_2)}{p(z_3-z_2)}z_1+\frac{p(z_2-z_1)}{p(z_3-z_2)}z_3\quad\mbx{and}\quad p(z_3-z_1)=p(z_2-z_1)+p(z_3-z_2)\,.\eequs
\item[\rm 3.] Let $\W$ be a convex subset of $X$, $Y$ a vector space ordered by a cone $C$  and $f:\W\to Y$ a $C$-convex function. For   $x_0:=x+t_0(y-x)\in D(x,y)\cap \W$, the $p$-slope of $f$ is given by
    $$
   \Delta_{p,x_0}(f)(z_t)=\frac{f(z_t)-f(x_0)}{p(z_t-x_0)}\,,$$
for $t\in\Real$ such that $z_t\in D(x,y)\cap\W\sms \{x_0\}$.

  Then $t_0<t<t'$\, or\, $t<t'<t_0$ implies
\bequ\label{eq1.slope}
 \Delta_{p,x_0}(f)(z_t)\le \Delta_{p,x_0}(f)(z_{t'})\,,\eequ
and  $t<t_0<t'$ implies
\bequ\label{eq2.slope}
 \frac{f(x_0)-f(z_t)}{p(x_0-z_t)}\le \frac{f(z_{t'})-f(x_0)}{p(z_{t'}-x_0)}\quad{\rm (} \iff -\Delta_{p,x_0}(f)(z_t)\le \Delta_{p,x_0}(f)(z_{t'}){\rm )}\; \,.\eequ
  \een
\end{prop}
\begin{proof}
  The equality from 1 follows by the definition of $z_t$.

  For 2, observe that the equality
  $$
  t_2=\frac{t_3-t_2}{t_3-t_1}\,t_1+ \frac{t_2-t_1}{t_3-t_1}\,t_3 $$
implies

$$
  z_2=\frac{t_3-t_2}{t_3-t_1}\,z_1+ \frac{t_2-t_1}{t_3-t_1}\,z_3 \,.$$

By 1,
$$
\frac{t_3-t_2}{t_3-t_1}=\frac{p(z_3-z_2)}{p(z_3-z_1)}\quad \mbx{and}\quad
\frac{t_2-t_1}{t_3-t_1}=\frac{p(z_2-z_1)}{p(z_3-z_1)}\,,$$
proving the representation formula for $z_2$.

The equality $p(z_3-z_1)=p(z_2-z_1)+p(z_3-z_2)$ is equivalent to $t_3-t_1=(t_3-t_2)+(t_2-t_1)$.

3. Let $x_0=x+t_0(y-x),\, z=x+t(y-x)$ and $z'=x+t'(y-x)$. The function $\vphi(t)=f(x+t(y-x))$ is convex, so that, by Proposition \ref{p.cv-fcs2},
its slope    is increasing.  If  $t_0<t<t'$, then

$$\frac{f(z)-f(x_0)}{p(z-x_0)}=\frac{\vphi(t)-\vphi(t_0)}{(t-t_0)p(y-x)}\le\frac{\vphi(t')-\vphi(t_0)}{(t'-t_0)p(y-x)}
=\frac{f(z')-f(x_0)}{p(z'-x_0)}\,.$$

The case $t<t'<t_0$  can be treated similarly. If  $t<t_0<t'$, then
$$\frac{f(x_0)-f(z)}{p(x_0-z)}=\frac{\vphi(t_0)-\vphi(t)}{(t_0-t)p(y-x)}\le\frac{\vphi(t')-\vphi(t_0)}{(t'-t_0)p(y-x)}
=\frac{f(z')-f(x_0)}{p(z'-x_0)}\,.$$
\end{proof}

\section{Continuity properties of convex functions}
In this section we prove some results on the continuity of convex functions.

We start with real-valued function of one real variable, a typical case. Based on the  monotonicity of the slope one can give a simple proof of the Lipschitz continuity of convex functions.
\begin{prop}\label{p1.cv-Lip} Let $\vphi:I\to\Real$ be a convex function defined on an interval $I\subset \Real$.
  Then   $\vphi$ is continuous on $\inter(I)$ and Lipschitz on every compact interval $[a,b]\subset\inter(I)$.
 \end{prop}\begin{proof}
   It is obvious that it suffices to check the fulfillment of the Lipschitz condition.
For  $[\alpha,\beta]\subset\inter(I)$ with $\alpha<\beta$, let $a,b\in\inter(I)$ be such that $a<\alpha <\beta<b$.

Let $\alpha\le t<t'\le\beta.$ By  Proposition \ref{p.cv-fcs2}.2,
$$
\frac{\vphi(t')-\vphi(t)}{t'-t}\le \frac{\vphi(b)-\vphi(t)}{b-t}\le\frac{\vphi(b)-\vphi(\beta)}{b-\beta}=:B\,,
$$
and
$$
A:=\frac{\vphi(\alpha)-\vphi(a)}{\alpha-a}\le\frac{\vphi(t')-\vphi(a)}{t'-a}\le \frac{\vphi(t)-\vphi(t')}{t-t'} \,.$$

It follows $|\vphi(t)-\vphi(t')|\le L\, |t-t'|,$ for all $t,t'\in[\alpha,\beta],$ where $L:=\max\{|A|,|B|\}$.
 \end{proof}

We mention also the following   properties of convex functions.

\begin{prop}\label{p.cv-aff} Let $I$ be an interval in $\Real, \, \varphi :I\to \Real$ a convex function and $a<b$ two points in $I$.
\ben
\item[\rm 1.] If for some
$0<t_0<1,\,  \varphi ((1-t_0)a+t_0b)=(1-t_0)\varphi (a)+t_0\varphi (b), $ then $\varphi $ is an affine function on the interval $[a,b],$
that is, $   \varphi ((1-t)a+tb)=(1-t)\varphi (a)+t\varphi (b) $  for every $t\in [0,1].$
\item[\rm 2.] Let $a,b\in\inter(I),\, a<b$. If $\vphi(a)<\vphi(b)$, then $\vphi$ is strictly inreasing on the interval $I_{b+}=\{\alpha\in I : \alpha\ge b\}$.
If $\vphi(a)>\vphi(b)$, then $\vphi$ is strictly decreasing on the interval $I_{a-}=\{\alpha\in I : \alpha\le a\}$.
\item[\rm 3.] Any nonconstant convex function $\vphi:\Real\to\Real$ is unbounded, more exactly $\sup\vphi(\Real)=+\infty$.
\item[\rm 4.] Let $\vphi:[0,\infty)\to[0,\infty)$ be convex such that $\vphi(\alpha)=0\iff \alpha =0$. Then $\vphi$ is   strictly increasing and superadditive, that is,
$$
\vphi(\alpha+\beta)\ge \vphi(\alpha)+\vphi(\beta)\,,$$
for all $\alpha,\beta\in[0,\infty)$.

If $\vphi:[0,\infty)\to [0,\infty)$ is concave and $\vphi(\alpha)=0\iff \alpha =0$, then $\vphi$ is increasisng and subadditive, that is,
$$
\vphi(\alpha+\beta)\le \vphi(\alpha)+\vphi(\beta)\,,$$
for all $\alpha,\beta\in[0,\infty)$.
\een\end{prop}
\begin{proof} 1.\;
Suppose that for some $\,t,\,t_0<t<1,\,
\varphi (a+t(b-a))<\varphi (a)+t(\varphi (b)-\varphi (a)). $ Let $c=a+t_0(b-a)$ and
$c_t=a+t(b-a).$
It follows $0<t_0/t <1,\; c=a+\frac {t_0}{t}(c_t-a),\,$ and   %
\begin{align*}%
\varphi (c)=&\varphi (a)+t_0(\varphi (b)-\varphi (a))=\left(1-\frac{t_0}{t}\right)\varphi (a)+\frac {t_0}{t}[\varphi (a)+t(\varphi (b)-\varphi (a))]\\&>\left(1-\f{t_0}{t}\right)\varphi (a)+\f{t_0}{t}\varphi (c_t), %
\end{align*} %
in contradiction to the convexity of $f.$

The case $0<t<t_0$ can be treated similarly.

2. Suppose that $\vphi(a)<\vphi(b)$ and let  $\alpha>b$   be a point in $I$. Then,  by the monotonicity of the slope,
$$
\frac{\vphi(\alpha)-\vphi(b)}{\alpha-b}\ge \frac{\vphi(b)-\vphi(a)}{b-a}>0\;\Lra\;\vphi(\alpha)> \vphi(b)\,.$$

If $b<\alpha<\alpha'$   belong to  $I$, then $\vphi(\alpha)>\vphi(b)$, and applying the above reasoning to the points $b<\alpha<\alpha'$, it follows $\vphi(\alpha)<\vphi(\alpha').$

In the case $\vphi(a)>\vphi(b)$, a similar argument applied to points $\alpha\in I$ with $\alpha< a$ shows that $\vphi$ is strictly decreasing on $I_{a-}$.

3. \; Suppose that there exists two points  $a<b$ in $\Real$ such that $\vphi(a)\ne\vphi(b).$

\emph{Case} I. \; $\vphi(b)-\vphi(a)>0$

Let $\alpha_t=a+t(b-a),\, t>1$. The monotonicity of the slope implies
$$
\frac{\vphi(\alpha_t)-\vphi(a)}{\alpha_t-a}\ge \frac{\vphi(b)-\vphi(a)}{b-a}\,.$$

Since $\alpha_t-a=t(b-a)>0$, it follows
$$
\vphi(\alpha_t)-\vphi(a)\ge t\,(\vphi(b)-\vphi(a))\to +\infty \;\;\mbx{as}\;\; t\to\infty\,.$$

\emph{Case} II. \; $\vphi(b)-\vphi(a)<0$

Taking   $\alpha_t=a+t(b-a) $ for $t<0$, it follows   $\alpha_t<a<b$, so that,  by the monotonicity of the slope,
$$
\frac{\vphi(\alpha_t)-\vphi(a)}{\alpha_t-a}\le \frac{\vphi(b)-\vphi(a)}{b-a} \,.$$

Since, in this case, $\alpha_t-a=t(b-a)<0$, it follows
$$
\vphi(\alpha_t)-\vphi(a)\ge t\,(\vphi(b)-\vphi(a))\to +\infty \;\;\mbx{as}\;\; t\to -\infty\,.$$

4.\; By 2, $\vphi$ is strictly increasing on $[0,\infty)$  because  $\vphi(\alpha)>0=\vphi(0)$ for every $\alpha>0$.

Let now $0<\alpha< \beta$. Then, by the convexity of $\vphi$,
$$
\vphi(\alpha)=\vphi\left(\big(1-\frac\alpha\beta\big)0+\frac\alpha\beta\beta\right)\le \frac\alpha\beta \vphi(\beta)\,,$$
so that
\bequ\label{eq1.sup-ad}
\alpha\vphi(\beta)-\beta\vphi(\alpha)\ge 0\,.
\eequ

 Again, by the convexity of $\vphi$,
\begin{align*}
 \vphi(\beta)\le \frac\alpha\beta\vphi(\alpha)+\frac{\beta-\alpha}{\beta}\vphi(\alpha+\beta)\\
 \end{align*}
implying
\begin{align*}
  \vphi(\alpha+\beta) \ge\frac{\alpha\vphi(\beta)-\beta\vphi(\alpha)}{\beta-\alpha}+\vphi(\alpha)+\vphi(\beta)\overset{\eqref{eq1.sup-ad}}{\ge}\vphi(\alpha)+\vphi(\beta)\,.
 \end{align*}

   Suppose now that $\vphi$ is concave and  not increasing on $[0,\infty)$. Then there exist  two numbers $0<\alpha<\beta$ such that $\vphi(\alpha)>\vphi(\beta)$. Let  $\alpha_t=\alpha+t(\beta-\alpha)$  with $t>1$.   Since the slope of $\vphi$ is decreasing, we have

   $$
\frac{\vphi(\alpha_t)-\vphi(\alpha)}{\alpha_t-\alpha}\le \frac{\vphi(\beta)-\vphi(\alpha)}{\beta-\alpha}\,,$$
implying
$$
 \vphi(\alpha_t)\le-\vphi(\alpha)+t (\vphi(\beta)-\vphi(\alpha))\lra -\infty \quad\mbox{as}\;\ t\to \infty\,.$$

 Consequently,  $\vphi(\alpha_t)<0$ for $t$ large enough, in contradiction to the hypothesis that $\vphi\ge 0$.

 The proof of the subadditivity follows the same line (reversing the inequalities) as  the proof of superadditivity in the case of a convex function.
\end{proof}
\begin{remark}
  Geometrically, the property 1 from Proposition \ref{p.cv-aff} says that if a point $(t_0,\vphi(t_0))$, with $a<t_0<b$, belongs to the segment $[A,B]$ where $A(a,\vphi(a))$ and $B(b,\vphi(b))$ are points on the graph of $\vphi$, then the graph of $\vphi$ for $t\in[a,b]$ agrees with the segment $[A,B]$.

  The example of the function $\vphi(t)=t$ for $t\in[0,1]$ and $\vphi(t)=1$ for $t\ge 1$ shows that a concave function satisfying the hypotheses from Proposition \ref{p.cv-aff}.4, can be only increasing, not strictly.
\end{remark}

We consider now a more general situation.

\begin{prop}\label{cs-p.cv-cont} Let $X$ be a TVS, $\W\subset X$ open and convex and $f:\W\to \Real$  a convex function.
\begin{enumerate}
\item[{\rm 1.}] If the function $f$ is bounded from above on a neighborhood of some point $x_0\in \W$, then $f$ is continuous at $x_0.$
\item[{\rm 2.}] If there exists a point $x_0\in \W$ and a neighborhood $U\subset \W$ of $x_0$ such that $f$ is bounded from above on $U$, then $f$ is locally bounded from above on $\W$, that is, every point $x\in \W$ has a neighborhood $V\subset \W$ such that $f$ is bounded from above on $V.$
 \item[{\rm 3.}]    If the function $f$ is bounded from above on a neighborhood of some point $x_0\in \W$, then $f$ is continuous on $\W.$
\end{enumerate}
\end{prop}\begin{proof}
  1.\; Let $ U$  be a balanced neighborhood of 0  such that $x_0+U\subset \W$ and, for some $ \beta>0,$  $\,f(x)\leq \beta\,$ for all $x\in x_0+U,$
or, equivalently, to $f(x_0+u)\leq \beta$ for all $u\in U.$

 For $0<\epsic<1,\; \pm\epsic u\in U$ and, by the convexity of $f$,
 \bequs
 f(x_0+\epsic u)-f(x_0)=f((1-\epsic)x_0+\epsic(x_0+u))-f(x_0)\leq (1-\epsic)f(x_0) +\epsic f(x_0+u)-f(x_0),
 \eequs
 so that
 \bequ\label{cs-eq1.cv-cont1}
 f(x_0+\epsic u)-f(x_0)\leq \epsic(f(x_0+u)-f(x_0))\leq \epsic(\beta-f(x_0)).
 \eequ

On the other side
\bequs %
f(x_0)=f\left(\frac{x_0+\epsic u+x_0-\epsic u}{2}\right)\leq \frac{1}{2}f(x_0+\epsic u)+\frac{1}{2}f(x_0-\epsic u),
 \eequs
 implying
 \bequ\label{cs-eq2.cv-cont1}
 f(x_0)-f(x_0+\epsic u)\leq f(x_0-\epsic u)-f(x_0)\leq \epsic(\beta -f(x_0)).
\eequ

The last inequality from above follows by replacing $u$ with $-u$ in \eqref{cs-eq1.cv-cont1}.  Now, by
\eqref{cs-eq1.cv-cont1} and \eqref{cs-eq2.cv-cont1} it follows
$$
|f(x_0+\epsic u)-f(x_0)|\leq \epsic (\beta-f(x_0))\quad\mbox{for all}\quad u\in U,
$$%
which is equivalent to
$$
|f(x_0+v)-f(x_0)|\leq \epsic (\beta-f(x_0))\quad\mbox{for every}\quad v\in \epsic U,
$$%
which shows that $f$ is continuous at $x_0.$

2.\; The proof has a geometric flavor and can be nicely illustrated by a drawing. Let  $U$ be a balanced neighborhood of 0 such that
$x_0+U \subset \W$ and, for some  $\beta>0,$   $\, f(x)\leq \beta$ for all $x\in x_0+U.$

Let $x\in \W.$ Since the set $\W$ is open, there exists $\alpha>1$ such that $x_1:=x_0+\alpha(x-x_0)\in \W,$
implying $x=\frac{\alpha-1}{\alpha}x_0+\frac{1}{\alpha}x_1. $ Putting $t=1/\alpha$ it follows $x=(1-t)x_0+tx_1$ with $0<t<1$. Consider the neighborhood
 $V:=  x+(1-t)\,U$  of $x$. We have $V\subset \W, $ because, by the convexity of $\W,$
 $$
 x+(1-t)u=tx_1+(1-t)(x_0+u)\in
 t\W+(1-t)\W \subset \W,
 $$
  for all $u\in U. $

 Also
 \begin{align*}
 f\left(x+ (1-t)u\right)=&f\left(tx_1+(1-t)(x_0+u)\right)\leq tf(x_1)+
 (1-t) f(x_0+u)\\ \leq& tf(x_1)+(1-t)\beta,
\end{align*}
  for every $u\in U.$

  3.\; The assertion from 3 follows  from 1 and 2.
\end{proof}

Based on this results one can give a characterization of the continuity of a convex function in terms of its epigraph. Let $X$ be a vector space, $\W$ a nonempty subset of $ X$  and $f:\W\to\Real$ a function. Let
\begin{align*}
 \epi(f)= &\,\{(x,\alpha)\in X\times \Real : f(x)\le\alpha\}\quad\mbx{and}\\
\epi'(f)= &\,\{(x,\alpha)\in X\times \Real : f(x)<\alpha\}\,,
\end{align*}
be the epigraph and, respectively, the strict epigraph of $f$.

The following result is a direct consequence of the definitions.

\begin{prop}\label{p1.epigraph}Let $X$ be a vector space, $\W\subset X$ a convex set and $f:\W\to\Real$ a function.
The following equivalences hold:
\begin{align*}
 \mbx{the function } f \mbx{ is convex } &\iff \epi(f) \mbx{ is a convex subset of } X\times \Real\\
&\iff  \epi'(f) \mbx{ is a convex subset of } X\times \Real\,.
 \end{align*}\end{prop}

 We can characterize now the continuity of $f$.

 \begin{prop}\label{p2.epigraph}
   Let $X$ be a TVS, $\W\subset X$  nonempty open convex and $f:\W\to \Real$  a convex function.
   \ben
   \item[\rm 1.] {\rm(a)}\;\; $\inter(\epi(f))\subset \epi'(f)$;\\
   {\rm(b)}\; if $f$ is continuous at $x\in\W$, then $(x,\alpha)\in\inter(\epi(f)$ for all $\alpha>f(x)$;\\
   {\rm(c)} \; if $(x,\alpha)\in\inter(\epi(f)$, then $f$ is continuous at $x$.
\item[\rm 2.] The following are equivalent:\\
{\rm(i)}\; $f$ is continuous on $\W$;\\
{\rm(ii)}\; $\inter(\epi(f))\ne\emptyset$; \\
{\rm(iii)}\; $\epi'(f)$ is an open subset of $X\times \Real$.
\item[\rm 3.] If  $\inter(\epi(f))\ne\emptyset$, then  $\inter(\epi(f))=\epi'(f)$.
  \een\end{prop}
  \begin{proof} 1.(a)\; If $(x,\alpha)\in \inter(\epi(f))$, then there exist a neighborhood $U$ of $0\in X$ and $\delta>0$ such that
  $W:=(x+U)\times(\alpha-\delta,\alpha+\delta)\subset \epi(f)$. But then $(x,\alpha-\delta/2)\in W \subset \epi(f)$ so that $f(x)\le \alpha-\delta/2
  <\alpha$, that is, $(x,\alpha)\in\epi'(f)$.

  (b)\;
  Let $W\subset\epi (f)$  be as above. Then, for every $u\in U, \; (x+u,\alpha)\in W\subset \epi(f)$, so that $f(x+u)\le \alpha$ for all $u\in U$,
  which,   by Proposition \ref{cs-p.cv-cont}, implies the continuity of $f$ at $x$.

   (c)\; Suppose that $f$ is continuous at $x\in \W$ and let $\alpha>f(x)$. Then $\delta:=(\alpha-f(x))/2>0$ and there exists a neighborhood $U$ of $0\in X$ such that
   $$
  f(x+u)<f(x+\delta)=\alpha-\delta<\alpha\,,$$
  for all $u\in U$.  It follows that the neighborhood $(x+U)\times (\alpha-\delta,\infty)$ of $(x,\alpha)$ is contained in $\epi(f)$, which implies that $(x,\alpha)\in\inter(\epi(f))$.

  2. Notice that, by Proposition \ref{cs-p.cv-cont}, the continuity of $f$ at a point $x\in\W$ is equivalent to the continuity of $f$ on $\W$.

  (i)$\iff$(ii) follows from the assertions (b) and (c) of point 1 of the proposition.

  (i) $\Ra$ (iii).

  Suppose that $f$ is continuous on $\W$. If $(x,\alpha)\in\epi'(f)$, then $f(x)<\alpha$ so that, by 1, (b) and (a), $(x,\alpha)\in\inter(\epi(f))\subset \epi'(f)$.
It follows that $\inter(\epi(f))$ is a neighborhood of $(x,\alpha)$ contained in $\epi'(f)$, that is, $(x,\alpha)\in\inter(\epi'(f))$. Consequently
$\epi'(f)\subset\inter(\epi'(f))$ and so $\epi'(f)=\inter(\epi'(f))$ is open.

(iii) $\Ra$ (i)

If $\epi'(f)$ is open, then,   $\emptyset\ne\epi'(f)\subset \inter(\epi(f))$ so that (ii) holds, which implies the continuity of $f$.

3. If  $\inter(\epi(f))\ne \emptyset$, then $f$ is continuous on $\W$, so that $\epi'(f)$ is open. The inclusion $\epi'(f)\subset \inter(\epi(f))$ implies
$\epi'(f)\subset \inter(\epi(f))$ and so, taking into account 1.(a), $\epi'(f)= \inter(\epi(f))$.
\end{proof}

The following proposition shows that in the finite dimensional case the convex functions are continuous.

\begin{prop}\label{cs-p.cv-cont-Rn} Let  $f:\W\subseteq \mathbb{R}^{n}\rightarrow
\mathbb{R}$  be a convex function, where the set $\W$  is
open and convex. Then  $f$  is locally bounded from above on $\W$.

Consequently, $f$ is continuous  on $\W.$
\end{prop}
\begin{proof} Let us choose $x_{0}\in \W$ and $K\subseteq \W$ be a hypercube
having the center in $x_{0}$.

We are going to prove that $f$ is bounded from above on $K$.

If $v_{1},...,v_{m}$, where $m=2^{n}$, are the vertices of $K$, then
  for each $x\in K$ there exist $%
\lambda _{1},...,\lambda _{m}\in [0,1],\, \overset{m}{\underset{k=1}{%
\sum }}\lambda _{k}=1$, such that $x=\overset{m}{%
\underset{k=1}{\sum }}\lambda _{k}v_{k}.$

On one hand, taking into account Jensen's inequality for convex functions,
we obtain that
\begin{equation*}
f(x)=f(\overset{m}{\underset{k=1}{\sum }}\lambda _{k}v_{k})\leq
\sum_{k=1}^{m}\lambda _{k}f(v_{k})\leq \max_{k\in \{1,2,...,m\}}f(v_{k})\,,
\end{equation*}%
showing that  $f$ is bounded from above on $K$.
 \end{proof}

 A convex function defined on an infinite dimensional normed linear space is
not necessarily locally bounded as the following example shows.

\begin{example}\label{ex.3.1} Let $X$\ be the space of polynomials endowed with the\
norm given by%
\begin{equation*}
\left\Vert P\right\Vert =\max\limits_{x\in \lbrack -1,1]}\left\vert
P(x)\right\vert \text{.}
\end{equation*}

Then the function\textit{\ }$f:X\rightarrow \mathbb{R}$\textit{\ }given by\
\begin{equation*}
f(P)=P^{\prime }(1)
\end{equation*}%
for each\textit{\ }$P\in X$ is convex (even linear) but it is not locally bounded.\end{example}

 Consider for each    $n\in \mathbb{N}$ the polynomial
\begin{equation*}
P_{n}(x)=\frac{1}{\sqrt{n}}x^{n}\text{.}
\end{equation*}

Then
\begin{equation*}
\|P_n\|=\frac1{\sqrt n}\to 0,\quad n\to \infty,
\end{equation*}%
but
\begin{equation*}
  f(P_{n}) =\sqrt n\to \infty,\quad  n\to \infty\,,
\end{equation*}%
proving the discontinuity of the functional $f$.

\begin{remark}\label{re.cont-cv-fcs} In fact a normed space $X$ is finite dimensional if and only if every linear functional on $X$ is continuous. On the other hand there exists infinite dimensional locally convex spaces $X$ such that every convex function on $X$ is continuous.
\end{remark}

Indeed, it is known that every linear functional on a finite dimensional topological vector space is continuous. If $X$ is an infinite dimensional  normed space then it contains a linearly independent set $D=\{e_n : n\in\Nat\}\subset S_X.$ Consider a Hamel basis $E$ of $X$ containing this set and define $\vphi:E\to \Real$ by $\vphi(e_n)=n,\,n\in\Nat,$ and $\vphi(e)=0$ for $e\in E\sms D,$ extended   by linearity to whole $X$. Then $\sup\{\vphi(x) : x\in X,\, \|x\|\le 1\}\ge \sup\{\vphi(e_n) : n\in\Nat\}=\infty,$ proving the discontinuity of $\vphi.$

Concerning the second affirmation, let $X$ be an infinite dimensional vector space   equipped  with the finest locally convex topology $\tau$. A  neighborhood basis at 0 for  this topology is formed by all absolutely convex absorbing subsets of $X$. A family of seminorms generating this topology is formed of the Minkowski functionals of these neighborhoods. Since every seminorm $p$ on $X$ is the Minkowski functional of the absolutely convex absorbing subset $B_p=\{x\in X: p(x)\le 1\}$, it follows that $\tau$ is generated by the family $P$ of all seminorms on $X$. It is in fact characterized by this property: the finest locally convex topology on a vector space $X$ is the locally convex topology $\tau$ on $X$ such that  every seminorm on $X$ is $\tau$-continuous. For  the finest locally convex topology on a vector space,  see \cite[p. 56 and Exercise 7, p. 69]{Schaef-TVS} and \cite[pp. 3--4]{Bonet}. It follows that every convex absorbing subset of $X$ is a neighborhood of 0 and  every linear functional is continuous on $X$. Also  every convex function defined on a nonempty open convex subset $\W$ of  $X$ is   continuous on $\W$.

For the convenience of the reader we sketch the proof following   \cite{co-mun76}, where further details can be found.

\smallskip
\emph{Fact} 1. \emph{If $C$ is a convex subset of vector space such that $ 0\in C$, then $\alpha C\subset \beta C$
for all $0<\alpha<\beta$.}

\smallskip

Indeed, by the convexity of $C$ and the fact that $0\in C$,
$$
\alpha c =\beta\left(\frac\alpha\beta c+\left(1-\frac\alpha\beta\right)\cdot 0\right)\in \beta C\,,$$
for all   $c\in C$.
\smallskip

\emph{Fact} 2. \emph{Let $Y$ be a vector space  equipped with the finest local convex topology $\tau$. Then every  convex absorbing subset $C$ of $Y$ is a neighborhood of} 0.
\smallskip

The set $D:=C\cap(-C)$  is absolutely convex and contains 0. For $x\in Y$ there exist $\alpha,\beta>0$ such that $x\in\alpha C$ and $-x\in \beta C\iff x\in\beta(-C)$.
Then, by Fact 1, $x\in \gamma C\cap \gamma (-C)$, where $\gamma=\max\{\alpha,\beta\}$. This implies that there exist $c,c'\in C$ such that $x=\gamma c$ and $x=\gamma (-c')$. But then $c=-c'\in -C$, that is, $x\in\gamma D$. Since $D$ is absolutely convex and absorbing  it  is a neighborhood of 0 as well  as $C\supset D$.

\smallskip
\emph{Fact} 3. \emph{Let $X$ be a vector space. Consider the space $X\times\Real$ equipped with the finest locally convex topology and $X$ with the induced topology. If   $\W$ is an open convex subset of  $X$, then every convex function $f:\W\to \Real$ is continuous}.
\smallskip

For more clarity we denote by $\theta$ the null element in $X$.

We can suppose, passing, if necessary, to the set $\wtil \W:=\W-x_0$ and to the function $\tilde f(x):=f(x+x_0)-f(x_0)-1$,\,  $x\in\wtil\W$, that $\theta\in \W$ and $f(\theta)<0$.

The convex function $f$ is continuous on $\W$ if and only if it is continuous  at $\theta\in\W$. In its turn, by Proposition \ref{p2.epigraph},  this holds if  the strict epigraph $\epi'(f):=\{(x,\alpha)\in X\times \Real : f(x)<\alpha\}$ is a neighborhood of $(\theta,0)$ in $X\times \Real$.    By Fact 2, $\epi'(f)$ is a neighborhood of $(\theta,0)$ in $X\times \Real$ if  it is convex and absorbing in $X\times \Real$.

The convexity of $\epi'(f)$ follows from the convexity of $f$.

Let us  show that $\epi'(f)$ is absorbing. Consider first the case $(\theta,\alpha)\in X\times \Real$. If $\alpha>f(\theta)$, then  $(\theta,\alpha)\in\epi'(f)$. If $\alpha\le f(\theta)<0$, then, as $\lim_{\gamma\searrow 0}\gamma\alpha=0$, it follows $\gamma \alpha>f(\theta) $ for sufficiently small positive $\gamma$, that is,  $\gamma(\theta,\alpha)=(\theta,\gamma\alpha)\in\epi'(f)$. Let now $(x,\alpha)\in X\times \Real$ with $x\ne \theta$. Then $I:=\{t\in\Real : tx\in\W\}$ is an open interval in $\Real$ and $g:I\to\Real,\, g(t):=f(tx),\, t\in I,$ is convex, and so continuous. But then $\epi'(g)$ is an open convex subset of $\Real^2$. Since $g(0)=f(\theta)<0$, it follows that $(0,0)\in\epi'(g)$, hence, by Proposition \ref{p2.epigraph}, $\epi'(g)$ is a neighborhood of $(0,0)$, and so an absorbing set in $\Real^2$. Let $\lbda>0$ be such that $(\lbda,\lbda\alpha)=\lbda(1,\alpha)\in\epi'(g)$. The equivalences
\begin{align*}
 (\lbda,\lbda\alpha)\in\epi'(g)&\iff  g(\lbda)<\lbda\alpha\iff f(\lbda x)<\lbda\alpha\\&\iff
 \lbda(x,\alpha)=(\lbda x,\lbda\alpha)\in\epi'(f)\,,\end{align*}
 show that  $\lbda(x,\alpha)\in\epi'(f)$ and so $\epi'(f)$ is an absorbing subset of $X\times \Real$.

  \section{Some  further properties of  convex vector-functions}\label{S.cv-vect-fcs}

Now we shall present, following \cite{papag83a}, some further results on $C$-convex mappings.

Let $X$ be a TVS, $Y$ a vector space ordered by a cone $C$ and $\W$ an open subset of $X$.
We say that a mapping $f:\W\to Y$ is \emph{locally $o$-bounded} on $\W$ if every point in $\W$ has a neighborhood on which $f$ is $o$-bounded.

The following proposition is the analog of  Proposition \ref{cs-p.cv-cont} with boundedness replaced by $o$-boundedness.
\begin{prop}\label{p1.o-bd-cv}
  Let $X,Y$ be as above and suppose that $\W\subset X$ is open and convex and $f:\W\to Y$ a $C$-convex mapping.
   \ben
   \item[\rm 1.] If $f$   upper $o$-bounded on a neighborhood of some point $x_0\in\W$, then $f$ is locally $o$-bounded on $\W$.
   \item[\rm 2.] If $Y$ is a TVS  ordered by a  normal cone $C$ and $f$ is $o$-bounded on a neighborhood of a point $x_0\in\W$, then  $f$ is continuous at $x_0.$
    \item[\rm 3.] If $Y$ is a TVS  ordered by a  normal cone $C$ and $f$ is  upper $o$-bounded on a neighborhood of some point $x_0\in \W$, then $f$ is continuous on $\W$.
 \een   \end{prop}\begin{proof} 1.\;
    Let $U$ be a balanced 0-neighborhood and let $y\in Y$ be such that $x_0+U\subset \W$ and $f(x_0+u)\le y$ for all $u\in U$.
    Then $-u\in U$ and
    $$
    f(x_0)\le \frac12[f(x_0+u)+f(x_0-u)]\, $$
    implies
    $$
    f(x_0)-f(x_0+u)\le f(x_0-u)-f(x_0)\le y-f(x_0)\,.$$

    It follows
    $$f(x_0+u)\ge 2f(x_0)-y\,,$$
   for all $u\in U$,  showing that $f$ is also lower $o$-bounded on $x_0+U$.

   The fact that $f$ is locally $o$-bounded on $\W$ can be proved similarly to the proof of assertion 2 in Proposition \ref{cs-p.cv-cont}.

   2.\; Suppose first that $0\in\W$ and $f(0)=0$.  Let $U\subset \W$ be a balanced  neighborhood of 0 such that $f$ is $o$-bounded on $U$, that is, the set $f(U)$ is $o$-bounded in $Y$.  Since the cone $C$ is normal it follows that $f(U)$ is topologically bounded. Let $V$ be a balanced $C$-full neighborhood of $f(0)=0\in Y$.
   The boundedness of $f(U)$ implies the existence of $\lbda>0$ such that $\lbda f(U)\subset V$. Since $V$ is balanced we can suppose further that $\lbda <1$.

   By the convexity of $f$
   $$
   f(\lbda u)=f((1-\lbda)0+\lbda u)\le (1-\lbda)f(0)+\lbda f(u)=\lbda f(u)\in V\,,$$
   for all $u\in U$.

   Also
   $$
   0=f(0)\le\frac 12[f(-\lbda u)+f(\lbda u)]$$
   implies
   $$ f(\lbda u)\ge -f(-\lbda u)=-f(\lbda (-u))\ge -\lbda f(-u)\in V\,.$$

   Consequently, $\,-\lbda f(-u)\le f(\lbda u)\le \lbda f(u),\,$ with  $\,-\lbda f(-u),  \lbda f(u)\in V.$  Since $V$ is $C$-full, this implies $f(\lbda u)\in V$ for all $u\in U$. Since $\lbda U$ is a neighborhood of $0\in X$ and $f(\lbda U)\subset V$, this proves the continuity of $f$ at $0$.

   In general, for $x_0\in \W$ consider the set $\wtil W=-x_0+\W$ and the function $\wtil f:\wtil\W \to Y$ given by $\wtil f(z)=f(x_0+z)-f(x_0)$. It follows that $\wtil f$  is $o$-bounded on a neighborhood $U\subset \wtil \W$ of $0\in X$, so that it is continuous at 0, implying the continuity of the mapping $f$ at $x_0\in\W$.

      The assertion from 3 follows from 1 and 2.
   \end{proof}

   In the finite dimensional case one obtains the following extension of Proposition \ref{cs-p.cv-cont-Rn}.

   \begin{corol}\label{c1.o-cont-Rn}
   Let $\Om$ be a nonempty open convex subset of $\Real^n$ and $Y$ a TVS ordered by a normal cone $C$. Then every $C$-convex function
   $f:\Om\to Y$ is locally $o$-bounded, and so continuous, on $\Om$.
     \end{corol}\begin{proof}
       The proof of Proposition \ref{cs-p.cv-cont-Rn} can be transposed \emph{mutatis mutandis} to this situation, replacing the order relation in $\Real$ by the order relation $\le_C$ generated by the normal cone $C$.
     \end{proof}

  Carioli and  Vesel\'y \cite{vesely13} showed that the normality of the cone $C$ is, in some sense,  necessary for the continuity of upper $o$-bounded convex vector-functions.
   \begin{theo}\label{t.Vesely}
     Let $I\subset\Real$ be an open interval, $X$ a (nontrivial) locally convex space,
$\W\subset  X$ an open, convex set and $Y$ a  Banach space ordered by a closed  cone $C$.
 The following assertions are equivalent.
 \ben
\item[\rm 1.]  The cone $C$ is normal.
\item[\rm 2.]  Every convex function $\vphi: I\to  Y$ is continuous.
\item[\rm 3.]  Every convex function $\vphi: I \to Y$ is locally norm bounded.
\item[\rm 4.] Every convex function $f : \W\to Y, $ which is upper $o$- bounded  on some open
subset of $\W$, is continuous.
\item[\rm 5.] Every convex function $f:\W\to Y$, which is upper $o$- bounded on some nonempty
open subset of $\W$, is locally norm bounded.
  \een \end{theo}

  The proof follows the following steps.

  \smallskip

  Step 1. \emph{If $\,Y$ is a Banach space ordered by a closed cone $C$ which is not normal, then there exists $w\ge 0$ in $Y$ such that
 the order interval $[0,w]_o$ is norm-unbounded}.

 \smallskip

 Since $C$ is not normal there exist two sequences $(x_n)$ and $(y_n)$ in $Y$ such that $0\le x_n\le y_n,\, \|y_n\|=1$ and $\|x_n\|=3^n$. One takes
 $w=\sum_{k=1}^\infty2^{-k}y_k$ and $$z_n=w-\sum_{k=1}^{n-1}2^{-k}y_k-2^{-n}x_n=w-\sum_{k=1}^{n}2^{-k}y_k+2^{-n}(y_n-x_n).$$ Then
 $0\le z_n\le w$ and
 $$
 \|z_n\|\ge \left(\frac32\right)^n-\big\|w-\sum_{k=1}^{n-1}2^{-k}y_k\big\|\lra\infty\;\mbx{ as }\;\;n\to\infty\,.$$

 \smallskip

  Step 2. \emph{Let $Y$ and $C$ be as in Step 1.  Then there exists a continuous convex function $\vphi:\Real\to Y$ locally  upper $o$-bounded on  $\Real$ which is norm-unbounded on every neighborhood of 0}.

 \smallskip
 Let $[0,w]_o$ the norm-unbounded interval given by Step 1. Then the interval $[\alpha w,\beta w]_o$ is also norm-unbounded for every $0\le \alpha<\beta$. Take the numbers $\lbda,\alpha\,$ with $\, \lbda\in (0,1)\,$ and $1<\alpha<\lbda^{-1}(1-\lbda+\lbda^2)$.  Since $1-\lbda+\lbda^2>\lbda,\,\alpha$ is well defined. Consider the intervals $\Delta_n:=\left[\lbda^{2n}w,\alpha\lbda^{2n}w\right]_o$ for $n\in\Nat_0:=\Nat\cup\{0\}$. Since $\alpha\lbda<1-\lbda+\lbda^2<1$, it follows
 $\alpha\lbda^{2n+2}w\le \lbda^{2n}w\,$ and $\, \alpha\lbda^{2n+2}w\ne \lbda^{2n}w$, so that the intervals $\Delta_n$ are pairwise disjoint and $z'\le z$ for $z\in \Delta_n,\, z'\in \Delta_{n'}\,$ with $n<n'$.

 Choose $w_n\in\Delta_n$ such that $\|w_n\|>n$ and define the function $\vphi:\Real\to Y$ by $\vphi(t)=0$ for $t\in(-\infty,0]$, $\vphi(\lbda^k)=w_k,\,k\in\Nat_0,$ and affine on each interval $[\lbda^{n+1},\lbda^n].$  Then $\vphi(t)=\vphi_n(t)$ for $t\in [\lbda^{n+1},\lbda^n]$, where

 $$
 \vphi_n(t)=\frac{\lbda^{n}w_{n+1}-\lbda^{n+1}w_{n}}{\lbda^{n}-\lbda^{n+1}}+\mu_n t\,,
\; \mbx{  with }\;
 \mu_n=\frac{w_{n}-w_{n+1}}{\lbda^{n}-\lbda^{n+1}}\,.$$

 Put also $\vphi(t)=\vphi_0(t)$ for $t>1$. One shows that $\mu_{n+1}\le \mu_n$ and that the so defined function $\vphi$  is $C$-convex. Since $\|\vphi(\lbda^n)\|=\|w_n\|\to \infty$, it is norm-unbounded on every neighborhood of $0\in\Real$. Since it takes values in $[0,w]_o$,
 it is $o$-bounded, and so locally $o$-bounded on $\Real$.

  \smallskip

 Step 3. \emph{Let $X$ be a nontrivial Hausdorff locally convex space, and  $Y$ and $C$  as in Step 1.  Then there exists a continuous convex function $f:X\to Y$ which locally  upper $o$-bounded on some neighborhood of $0$ and norm-unbounded on every neighborhood of 0}.

 \smallskip

 Let $[0,w]_o$ be the norm-unbounded interval given by Step 1 and  $\vphi:\Real\to Y$ the convex function given by Step 2. For a fixed element $v\in X\sms\{0\}$  there exists a continuous linear functional $x^*\in X^*$ such that $x^*(v)=1$. Define the function $f:X\to Y$ by $f(x)=\vphi(x^*(x)),\, x\in X.$ Then $f$ is convex,  continuous and
  $$
  \|f(\lbda^nv)\|=\|\vphi(\lbda^n\|=\|w_n\|\to \infty\;\mbx { as }\; n\to \infty\,.$$

  The function $f$ is order  bounded on every neighborhood $V_\epsic $ of $0\in X$ of the form $V_\epsic=\{x\in X :|x^*(x)|<\epsic\},\, \epsic >0 $.

\section{Lipschitz properties of convex vector-functions}

In this section we shall prove some results on Lipschitz properties for convex vector-functions, meaning convex functions with respect to a cone.

\subsection{Convex functions on locally convex spaces}

We define first Lipschitz functions between locally convex spaces.

\begin{defi}\label{def1.Lip-lcs}
Let  $(X,P)$\ and $(Y,Q)$ be  locally convex
spaces,  where $P,Q$ are directed families of seminorms generating their topologies,  and $A\subseteq X$. A function $f:A\rightarrow Y$\ is said to
satisfy the \textit{Lipschitz condition} (or that $f$ is a \emph{Lipschitz function})  if for each   $q\in Q$\   there exist   $p\in P$\ and $L=L_q\ge0$ such that
\begin{equation*}
q(f(x)-f(y))\leq L p(x-y)\text{,}
\end{equation*}%
for all $x,y\in A$.

The function $f$ is called \textit{locally Lipschitz} on $A$ if every point $x\in A$ has a neighborhood $V$ such that $f$ is Lipschitz on $V\cap A$
\end{defi}

\begin{remark} It is easy to check that the definition does not depend on the (directed) families of seminorms $P,Q$ generating the locally convex topologies on $X$ and $Y$, respectively.

\end{remark}

\begin{remark}\label{re.Lip-lcs-norm}  If $X$ and $Y$ are Banach spaces then the above definition
coincides with the standard definition (with respect to the metrics generated by the norms).

If $Y=\Kapa$, then $f:A\rightarrow \mathbb{R}$\ is
\textit{Lipschitz} if there exist $p\in  P$\ and $L>0$ such that
\begin{equation*}
\left\vert f(x)-f(y)\right\vert \leq Lp(x-y)\text{,}
\end{equation*}
for all $x,y\in A$.
\end{remark}

 The next theorem shows that continuous convex vector-functions defined on open convex subsets of locally convex spaces are locally Lipschitz.
For a seminorm $p$ on a vector space $X$ we use the notations
$$
B_p=\{x\in X : p(x)\le 1\}\quad\mbox{and}\quad B'_p=\{x\in X : p(x)< 1\}\,.$$

Arbitrary balls satisfy  the equalities
\begin{align*}
  B_p[x_0,r]:=&\,\{x\in X : p(x-x_0)\le r\}=x_0+r B_p\,,\;\;\mbx{and}\\
   B_p(x_0,r):=&\,\{x\in X : p(x-x_0)< r\}=x_0+r B'_p\,,
\end{align*}
for $x_0\in X$ and $r>0$

\begin{theo}\label{cs-t.cv-Lip-lcs}
 Let $\,(X,P),\;(Y,Q)\,$ be locally convex spaces, $C$ a normal
 cone in $Y$ and $\Omega$ an open convex subset of $X$.

 If $\,f:\Omega \to Y\,$ is a continuous convex mapping then $f$ is
 locally Lipschitz on $\Omega$.

     Furthermore, $f$ is Lipschitz on every compact subset of $\Omega.$
\end{theo}

We start with the following proposition,  the key tool in the proof of the theorem.

\begin{prop}\label{p2.cv-Lip-lcs}
 Let $X$ be a vector space, $x_0\in X,\, p$ a seminorm on $X,\,$  $Y$ a vector space ordered by a cone $C$ and let $q$ be the Minkowski functional of an absolutely convex $C$-full absorbing subset $W$ of $Y$.

 For  $R>0$ let $V=B_p[x_0,R]$
and let  $f:V\to Y$ be a $C$-convex function.

     If, for some $\beta>0,\, $   $q(f(x))\le \beta p(x)$ for all $x\in V$, then for every $0<r<R,$
 \bequ\label{eq2.le2.cv-Lip-lcs}
 q(f(x)-f(y))\le\frac{2\beta}{R-r}\,p(x-y)\,,
 \eequ
 for all $x,y\in B_p[x_0,r]$.
 \end{prop}

 We need  the following simple remark.
 \begin{lemma}[\cite{Breck}, Prop. 2.5.6]\label{le.Mink-fcs} Let $Y$ be a vector space ordered by a cone $C$. If $W$ is a $C$-full absolutely convex absorbing subset of $Y$ then the Minkowski functional $q$ of $W$ is  a seminorm, satisfying the condition
 \bequ\label{eq.Mink-fcs}
 q(y)\le\max\{q(x),q(z)\}\,,
 \eequ
 for all $x,y,z\in Y$ with $x\le y\le z$.
 \end{lemma}\begin{proof}
   Let $a:=\max\{q(x),q(z)\}$. Then, for every $\epsic>0$,  $q(x),q(z)<a+\epsic$, so, by the definition of the Minkowski functional, there exist $b,c\in (0,a+\epsic)$ such that $x\in bW\,$ and $ z\in cW$. Since $W$ is balanced,
   $$
   bW=(a+\epsic)\,\frac{b}{a+\epsic}W\subset (a+\epsic)W\,,$$
   and
   $$
   cW=(a+\epsic)\,\frac{c}{a+\epsic}W\subset (a+\epsic)W\,,$$
   implying $(a+\epsic)^{-1}x,\, (a+\epsic)^{-1}z\in W$. Since $W$ is $C$-full and
   $(a+\epsic)^{-1}x\le (a+\epsic)^{-1}y\le (a+\epsic)^{-1}z$ it follows $(a+\epsic)^{-1}y\in W$ or, equivalently, $y\in (a+\epsic)W$. But then $q(y)\le a+\epsic$. Since $\epsic>0$ was arbitrarily chosen, this implies
   $$
   q(y)\le a=\max\{q(x),q(z)\}\,.$$
   \end{proof}
 \begin{proof}[Proof of Proposition \ref{p2.cv-Lip-lcs}]

Let $x,y\in B_p[x_0,r],\, x\ne y$.

\emph{Case} I.\;  $p(x-y)=0$.

In this case  the line $D(x,y):=x+\Real(y-x)$ is contained in  $B_p[x_0,r]$.

Indeed,  for $z_t=x+t(y-x),$
$$
p(z_t-x_0)\le p(x-x_0)+|t|p(y-x)\le r,$$
for all $t\in\Real,$ proving that $D(x,y)\subset    B_p[x_0,r].$

 For $t>1$ let $z_t=y+t(x-y)$ and $z'_t=x+t(y-x)$. Then
$\,x=(1-t^{-1})y+t^{-1}z_t\,$ and
$\, y =(1-t^{-1})x+t^{-1}z_t',\,$
so that, by the convexity of $\,f,$
$$
f(x)\leq (1-t^{-1})f(y)+t^{-1}f(z_t)
$$
implying
\begin{equation}\label{e20}
 f(x)-f(y)\leq t^{-1}(f(z_t)-f(y)).
\end{equation}

Interchanging the roles of $x$ and $y$ one obtains
\begin{equation}\label{e21}
f(y)-f(x)\leq t^{-1} (f(z_t')-f(x))\iff f(x)-f(y)\geq t^{-1} (f(x)-f(z_t'))
\end{equation}

But then, by Lemma \ref{le.Mink-fcs},
$$
q(f(x)-f(y))\le \max\{t^{-1}q(f(z_t)-f(y)),t^{-1} q(f(x)-f(z_t'))\} \le\frac{2\beta}{t}.
$$

Letting $t\to \infty$, one obtains $q(f(x)-f(y))=0$.

\emph{Case} II.\; $p(x-y)>0$.

The function $\,\psi:\mathbb R\to \mathbb R\,$ defined by
$\, \psi(t) =p(x-x_0+ t(y-x)),\, t\in \mathbb R,\,$
is continuous and $\,\psi(0) =p(x-x_0)\leq r < R,\;
\psi(1) =p(y-x_0)\leq r < R.$

The inequality
$$
\psi(t)\geq |t| p(y-x)-p(x-x_0)
$$
shows that $\,\lim_{|t|\to \infty} \psi(t) =\infty,\,$ so that there are
$a<0\,$ and $b>1\,$ such that
$$
\psi(a)=R\quad \mbox{and} \quad \psi(b) = R.
$$

Putting $\,u:= x+ a(y-x)\,$ and $\, v:=x+b(y-x)\,$, it follows
$$
u-x=x-x_0+a(y-x)-(x-x_0)\quad\mbx{and}\quad v-y=x-x_0+b(y-x)-(y-x_0)\,,$$
so that
\bequ\label{eq.ineqs-psi}
p(u-x)\geq \psi(a) - p(x-x_0) \geq R-r\quad\mbx{and}\quad
 p(v-y)\geq \psi(b)-p(y-x_0) \geq R-r.
\eequ

Appealing to \eqref{eq2.slope}, it follows
\bequ\label{eq1.Lip-slope}
\frac{f(x)-f(u)}{p(x-u)}\le \frac{f(y)-f(x)}{p(y-x)}\le \frac{f(v)-f(y)}{p(v-y)}
\,.\eequ

  By hypothesis and the inequalities \eqref{eq.ineqs-psi},  $q((f(x)-f(u))/p(x-u))\le2\beta (R-r)^{-1}$ and $q((f(v)-f(y))/p(v-y))\le2\beta (R-r)^{-1}$, so that, by Lemma \ref{le.Mink-fcs},
$$
q\left(\frac{f(y)-f(x)}{p(y-x)}\right)\le \frac{2\beta}{R-r}\iff q(f(y)-f(x))\le \frac{2\beta}{R-r}\,p(y-x)\,.$$
\end{proof}

\begin{remark} If $Y=\Real$ the case $p(x-y) =0$ can be treated appealing to Proposition \ref{p.cv-aff}. Indeed, as we have seen, in this case $D(x,y)\subset B_p[x_0,r]$, so we can consider the convex function $\vphi:\Real\to\Real,\, \vphi(t)=f(x+t(y-x)),\, t\in\Real.$  By hypothesis the function $\vphi$ is bounded, so that
  by Proposition \ref{p.cv-aff}.2 it is constant. But then $f(x)=\vphi(0)=\vphi(1)=f(y).$
\end{remark}

\begin{proof}[Proof of Theorem \ref{cs-t.cv-Lip-lcs}] Suppose that $P$ is directed and that the seminorms in $Q$
are the Minkowski functionals of the members of a neighborhood base of $\,0\in Y\,$
formed of absolutely convex $C$-full sets (\cite[V.3.1]{Schaef-TVS}).

Let $\, x_0\in \Omega \,$ and $\,q\in Q.$ The continuity of $f$ at $x_0$ implies
the existence of a seminorm $\,p\in P\,$ and of $\,R>0\,$ such that
$\, V:=x_0+RB_p \subset \Omega\,$ and
$$
q(f(x))\leq 1 \qquad \forall x\in V.
$$

If $\,0<r<R\,$ then, by Proposition \ref{p2.cv-Lip-lcs},
$$
q(f(x)-f(y))\leq \frac{2}{R-r}\, p(x-y)
$$
for all $\, x,y \in x_0+rB_p.$

Let us show now that $f$ is Lipschitz on every compact subset $K$ of
$\,\Omega.\,$  Let $\,q\in Q\,$ be the Minkowski functional of a
$C$-full absolutely convex neighborhood of $\,0\in Y.$
By the first part of the proof,  for every $\,x\in K\,$ there are
$\,p_x \in P,\, L_x >0\,$ and $\,r_x>0\,$ such that
$\,U_x := x +r_x B'_{p_x} \subset \Omega\,$ and
$$
q(f(u)-f(v))\leq L_x p_x(u-v) \quad \forall u,v \in U_x.
$$
The compactness of $K$ implies the existence of a finite set
$ \{x_1,...,x_n\}\subset K\,$ such that
$$
K \subset \bigcup_{i=1}^n U_i,
$$
where $\,U_i =U_{x_i}$. Put $\,p_i=p_{x_i},\;r_i=r_{x_i},\;L_i= L_{x_i},\,$
and let $\,p\in P,\; p\geq p_{i},\; i=1,...,n\,$
and $\,L= \max \{L_{1},...,L_{n}\}$.  We show that
\begin{equation}\label{e26}
q(f(x)-f(y))\leq L p(x-y)
\end{equation}
for all $\,x,y \in K.$

Let $\, x,y\,$ be  distinct points in $K$. Suppose first that
$\,p(x-y)>0.$ If $\,i,j \in \{1,...,n\}\,$ are such that
$\,x\in U_i\,$ and $\,y\in U_j\,$ then, since these sets are open, there
exist
$\,a<0\,$ and $\,b>1\,$ such that
$$
u:=x+a(y-x) \in U_i \quad \mbox{and} \quad
v:=x+b(y-x) \in U_j.
$$

Now, by \eqref{eq2.slope},
$$
\frac{f(x)-f(u)}{p(x-u)}\le\frac{f(y)-f(x)}{p(y-x)}\le \frac{f(v)-f(y)}{p(v-y)}\,,$$
so that, by Lemma \ref{le.Mink-fcs},
$$
\frac{q(f(y)-f(x))}{p(y-x)}\le\max\left\{\frac{q(f(x)-f(u))}{p(x-u)},\frac{q(f(v)-f(y))}{p(v-y)}\right\}\le L\,.$$

If $\,p(x-y)=0\,$ then
$$
p(y-x_i)\leq p(y-x)+p(x-x_i)< r_i
$$
implying $\, x,y \in U_i\,$ and
$$
q(f(x)-f(y))\leq L_{i} p_{i} (x-y)\leq L p(x-y).
$$
\end{proof}

Taking into account  Proposition  \ref{cs-p.cv-cont-Rn} and  Theorem \ref{cs-t.cv-Lip-lcs}, one obtains the following consequence.

\begin{corol}\label{cs-c.cv-Lip-Rn} Let  $f:\W\subseteq \mathbb{R}^{n}\rightarrow
\mathbb{R}$  be a convex function, where the set $\W$  is
open and convex.
 Then $f$  is locally Lipschitz on $\W$ and Lipschitz on every compact subset of $\W$.
\end{corol}

\subsection{The order-Lipschitz property}

Papageorgiou (\cite{papag83a,papag83b}) considered a notion of Lipschitzness for convex vector functions related to the order. Let $X$ be a normed space and $Y$ a normed lattice, $\W\subset X$ and $f:\W\to Y$. One says that $f$ is $o$-Lipschitz on a subset $Z$ of $\W$ if there exists $y\ge 0$ in $Y$ such that
\bequ\label{def.o-Lip}
|f(z)-f(z')|\le y \|z-z'\|\,,\eequ
for all $z,z'\in Z$.

Notice that  an $o$-Lipschitz function is Lipschitz. Indeed, from \eqref{def.o-Lip},
\bequs
\|f(z)-f(z')\|\le \|b\| \|z-z'\|\,,\eequs
for all $z,z'\in Z$, because in a normed  lattice $|x|\le |x'|$ implies $\|x\|\le\|x'\|$.

\begin{theo}\label{t.o-Lip} Let $X$ be a normed space, $Y$ a normed lattice, $\W \subset X$ open and convex and $f:\W\to Y$ a function convex with respect to  the order of $Y$.
If $f$ is upper $o$-bounded on a neighborhood of a point $x_0\in\W$, then $f$ is locally $o$-Lipschitz on $\W$.
  \end{theo}

  The proof will follow from an analog of Proposition \ref{p2.cv-Lip-lcs}.
  \begin{lemma}\label{le1.o-Lip}
Under the hypotheses of Theorem \ref{t.o-Lip}, if $R>0$ is such that $V=B[x_0,R]\subset \W$ and, for some $z\ge 0$ in $Y$,
\bequ\label{eq1.o-Lip}
|f(x)|\le z\,,\eequ
for all $x\in V$, then for every $0<r<R$
\bequ\label{eq2.o-Lip}
|f(x)-f(y)|\le \frac{2z}{R-r} \|x-y\|\,,\eequ
for all $x,y\in U:= B[x_0,r]$.
  \end{lemma}\begin{proof}
  The proof is similar to that of Proposition \ref{p2.cv-Lip-lcs}, so we only sketch it.

  Let $x\ne y$ in $U$. Since $\|x-y\|>0$ we have to consider only Case 2 of the corresponding proof. Like there, let $\alpha<0$ and $\beta>1$ be such that
  \bequs
  \|x-x_0+\alpha(y-x)\|=R= \|x-x_0+\beta(y-x)\|\,.\eequs

 Let   $\,u:= x+ a(y-x)\,$ and $\, v:=x+b(y-x).$  Putting $\,p(\cdot)=\norm$ in the inequalities \eqref{eq.ineqs-psi}, one obtains

\bequ\label{eq2.ineqs-psi}\begin{aligned}
&\|u-x\|\geq  R-r\quad\mbx{and}\quad
&\|v-y\| \geq R-r\,.
\end{aligned}\eequ

Appealing to \eqref{eq2.slope}, it follows
\bequ\label{eq1.Lip-slope}
\frac{f(x)-f(u)}{\|x-u\|}\le \frac{f(y)-f(x)}{\|y-x\|}\le \frac{f(v)-f(y)}{\|v-y\|}
\,.\eequ

  By hypothesis and the inequalities \eqref{eq2.ineqs-psi},
   \begin{align*}
  \frac{|f(x)-f(u)|}{\|x-u\|}\le \frac{2z}{R-r} \quad\mbx{ and }\quad \frac{|f(v)-f(y)|}{\|v-y\|}\le\frac{2z}{R-r}\,,
    \end{align*}
     so that
$$\frac{|f(y)-f(x)|}{\|y-x\|}\le \frac{2z}{R-r}\iff |f(y)-f(x)|\le \frac{2z}{R-r}\,\|y-x\|\,.$$
    \end{proof}

  \begin{proof}[Proof of Theorem \ref{t.o-Lip}] By Proposition \ref{p1.o-bd-cv} the function $f$ is locally $o$-bounded on $\W$. Therefore, for any $x\in \W$ there exist $R>0$ and $y\ge 0$ such that \eqref{eq1.o-Lip} holds. By Lemma \ref{le1.o-Lip} the function $f$ satisfies \eqref{eq2.o-Lip}, that is, it is $o$-Lipschitz on $B[x,r]$,  for   every $r\in(0,R)$.
  \end{proof}
  \begin{remark}\label{re.o-Lip} We have used some properties of the order relations in a vector lattice    (see Section \ref{S.OVS}).
For instance at the end of the proof of Lemma \ref{le1.o-Lip} we have used the following property
  $$
  u\le v\le w\;\Ra\; |v|\le|u|\vee |w|\,$$
 (see the proof following the relations \eqref{eq1.v-latt}),  applied to the inequalities  \eqref{eq1.Lip-slope}  .
  \end{remark}

\section{Equi-Lipschitz properties of families of continuous convex mappings}

Let $(X,P), (Y,Q)\,$ be real locally convex spaces, where $P,Q$ are
directed families of seminorms generating the topologies,
$\Omega$ an open
convex subset of $X$ and $F$ a family of functions from $\Omega $ to
$Y$. The family $F$ is called {\it equi-Lipschitz}  on  a subset $A$ of
$\Omega$ if for every $\, q\in Q$ there are  $p=p_q \in P$ and
a number $L_q\geq 0$ such that
\begin{equation}\label{e'3.1}
q(f(x)-f(y)) \leq L_q p(x-y)
\end{equation}
for all $\,x,y \in A\,$ and all $\,f\in F.$
The family
$F$ is called {\it locally equi-Lipschitz} on $\Omega$ if each point
$\,x\in \Omega\,$ has a neighborhood $\,U_x\subset \Omega\,$ such that
$F$ is equi-Lipschitz on $\,U_x.$

The family $F$ is called \emph{pointwise bounded} on $\Omega$ if, for every
$q\in Q,$
\begin{equation}\label{e'3.2}
\sup \{q(f(x)): f\in F\}   < \infty
\end{equation}
holds for each $ \, x\in \Omega$.

A \emph{barrel} in a locally convex space $(X,P)$ is an absorbing absolutely convex and
closed subset. The locally convex space $ X$  is called \emph{barrelled}  if each barrel is a neighborhood of 0 in X. Any Baire LCS, hence any
complete semimetrizable LCS, is a barrelled space. Notice that   there exist barrelled locally convex spaces and barrelled normed spaces  that are not Baire, see  \cite[p. 100]{Bonet} and  \cite{saxon74}, respectively. An example of an incomplete normed space that is  Baire  was given by Libor Vesel\'y, see\\

\emph{http://users.mat.unimi.it/users/libor/AnConvessa/Baire-incompleto.pdf}\\

The following result was proved in \cite{jou-thi84}. The proof given here is adapted from \cite{cobz01a}.

\begin{theo}\label{t.3.1}
Let $(X,P)$ be a barrelled locally convex space, $(Y,Q)$ a locally
convex space ordered by a normal  cone $C$ and $\Omega$ an open
convex subset of $X$.

If $F$ is a pointwise bounded family of continuous convex functions from
$\Omega$ to $Y$ then $F$ is locally equi-Lipschitz on $\Omega.$

Furthermore, the family $F$ is equi-Lipschitz on every compact subset
of $\Omega.$
\end{theo}
\begin{proof}
  Suppose  that the seminorms in
$Q$ are the Minkowski functionals of members of  a  neighborhood basis $\rond B$ of $\,0\in Y\,$ formed  of absolutely
convex $C$-full sets.

Let $\, x_0 \in \Omega,\; W\in \mathcal B\,$ and let  $\, q\in Q$ be
the Minkowski functional of the set $W\in\rond B$. We show that there
are $\,p\in P,\; R>0\,$ and $\,\beta >0\,$ such that
$\,V:= x_0+RB_p \subset \Omega\,$
and
\begin{equation}\label{e'3.3}
q(f(x)) \leq \beta
\end{equation}
for all $\,x\in V\,$ and all $\, f\in F.$
Taking into account Proposition \ref{p2.cv-Lip-lcs}, the relation (\ref{e'3.3}) yields
that, for any $\,0<r<R,\,$ we have
$$
q(f(x)-f(y))\leq \frac{2\beta}{R-r} p(x-y)
$$
for all $ \,x,y \in x_0+rB_p\,$ and all $\, f\in F.$

Let
$$
B=\{u\in X: x_0+u\in \Omega \; \mbox{and} \;
f(x_0+u)-f(x_0) \in \frac{1}{2}W-C \quad \forall f\in F\}
$$
A simple verification shows that $B$ is a convex subset of $X$.
We show that $B$ is also absorbing. To this end let $\,x\in X\,$ and
let $\,\alpha >0\,$ be such that $\, x_0+\alpha x \in \Omega\,$
(possible since the set $\Omega$ is open). For any $\, t,\; 0<t<1,\,x_0+t\alpha x\in\W$ (since $\W$ is convex)  and
$$
f(x_0+t\alpha x) =
f((1-t)x_0+ t(x_0+\alpha x)) \leq (1-t) f(x_0) + t f(x_0+\alpha x)
$$
implying
\begin{equation}\label{e'3.x}
f(x_0+t\alpha x)-f(x_0) \leq t(f(x_0+\alpha x) - f(x_0)).
\end{equation}
Since the family $F$ is pointwise bounded there exists
$\,t_0,\; 0<t_0<1,\,$ such that
$$
 t_0(f(x_0+ \alpha x)- f(x_0)) \in\frac{1}{2} W
$$
for all $\, f\in F,$  so that by (\ref{e'3.x}),
\begin{align*}
 &f(x_0+t_0\alpha x) -f(x_0)=\\ &=\left[f(x_0+t_0\alpha x) -f(x_0)- t_0(f(x_0+ \alpha x)- f(x_0))\right]+ t_0(f(x_0+ \alpha x)- f(x_0))\in -C+\frac{1}{2} W\,,
\end{align*}
 for all $\, f\in F,$ showing that $t_0\alpha x\in B$. Consequently, the set $\overline{B}\,$
is a barrel in $X$ and, since $X$ is barrelled, $\overline{B}\,$
is a neighborhood of $\, 0\in X.$

Take $\,R>0\,$ and $\,p\in P\,$ such that
$\, V:=x_0+RB_p \subset x_0+\overline{B}.$  For $\, f\in F\,$ and
$\, u \in RB_p\subset \ov B,\,$ there exists a net $(u_i)_{i\in I}$ in $B$ converging to $u$. The relations $f(x_0+u_i)-f(x_0)\in2^{-1}W-C$ and   the continuity of $f$ imply

$$
f(x_0+u)-f(x_0)=\lim_i(f(x_0+u_i)-f(x_0)) \in \mbox{cl}(\frac{1}{2} W-C) \subset W-C\,.$$ \quad

Similarly
$$
f(x_0-u) - f(x_0) \in    W-C.
$$

By the convexity of $f$
\begin{align*}
 2f(x_0)\le f(x_0+u)+f(x_0-u)\iff& f(x_0+u)-f(x_0)\ge f(x_0)-f(x_0-u) \\
 \Lra\;& f(x_0+u)-f(x_0)\in f(x_0)-f(x_0-u)+C\,.
\end{align*}

But then
$$
f(x_0+u)-f(x_0) \in -W+C+C=W+C\,.
$$
Therefore
$$
f(x_0+u)-f(x_0) \in (W-C)\cap (W+C)=W \subset B_q\,,
$$
i.e.
$$
q(f(x)-f(x_0)) \leq 1   \quad \forall x\in V \;\mbox{and}\; \forall f\in F\,.
$$

Hence
$$
q(f(x)) \leq 1+ q(f(x_0)) \leq 1+\sup \{ q(f(x_0)) : f\in F\}=:\beta.
$$
for all $\, x\in V\,$ and all $\, f\in F.$

The proof of the fact that $F$ is equi-Lipschitz on every compact subset of
$\Omega$ proceeds like in the case of one function, taking into account that, by \eqref{e'3.3},
we can add "for all $f\in F$" to each of the relations used in the proof of the
corresponding assertion of Theorem \ref{cs-t.cv-Lip-lcs}.
\end{proof}

\section{Convex functions on metrizable TVS}

In this section we shall discuss the Lipschitz properties of convex functions on metrizable TVS.

 As it was shown in \cite{co-mun76} continuous convex functions are also locally Lipschitz with respect to some  translation invariant metrics.

 For $0<p<1$ consider the linear space $\ell^p$ of all sequences $x=(x_k)$ of real numbers such that  $\sum_{k=1}^\infty|x_k|^p<\infty$.  The function

 \bequs
d(x,y)= \sum_{k=1}^\infty|y_k-x_k|^p
\eequs
is a translation invariant (i.e. $d(x+z,y+z)=d(x,y),\,\forall x,y,z\in X$) metric on $\ell^p$ generating a linear topology on $\ell^p$.

 \begin{prop}\label{cs-p.cv-Lip-elp}
 Let $\Om$ be an open convex subset of the space $\ell^p,\, 0< p<1.$ If $f:\Om\to \Real$ is continuous and convex, then $f$ is locally Lipschitz on $\Om.$
    \end{prop}\begin{proof}
      For $x_0\in\Om$ there exists $r>0$ and $a>0$ such that $|f(x)|\le a$ for all $x\in U,$ where $U:=\{x\in \ell^p : d(x_0,x)\le r\}\subset \Om$ is a neighborhood of $x_0.$  Let $V:=\{x\in \ell^p : d(x_0,x)\le r/4\}\subset U$.
For $x,y\in V,\, x\ne y,$ we have $d(x,y)\le r/2$ and
\begin{align*}
  &d\left(\frac r{2d(x,y)}\,(y-x),0\right) =\left(\frac r{2d(x,y)}\right)^p d(y-x,0)\\
  &=\left(\frac r{2d(x,y)}\right)^p d(x,y)=\left(\frac r{2}\right)^p \left(d(x,y)\right)^{1-p}\le\frac r2\,.
\end{align*}

     The element $z:=y+r\left(d(x,y)\right)^{-1}(y-x)$ belongs to $U$ because
      $$
      d(z-x_0,0)\le d(y-x_0,0)+d\left(\frac r{2d(x,y)}(y-x),0\right)\le\frac r4 + \frac r2<r.
      $$

It follows
$$
y=\frac{2d(x,y)}{2d(x,y)+r}z+ \frac{r}{2d(x,y)+r}x\,,
$$
so that, by the convexity of $f$,
$$
f(y)\le\frac{2d(x,y)}{2d(x,y)+r}f(z)+ \frac{r}{2d(x,y)+r}f(x)\,,
$$
implying
$$
f(y)-f(x)\le\frac{2d(x,y)}{2d(x,y)+r}\left(f(z)-f(x)\right)\le \frac{4a}{2d(x,y)+r}\,d(x,y)\le \frac{4a}{r}\,d(x,y)
$$

By symmetry
$$
f(x)-f(y)\le\frac{4a}{r}\,d(x,y)\,,
$$
so that
$$
|f(y)-f(x)|\le\frac{4a}{r}\,d(x,y)\,.
$$

Consequently $f$ is Lipschitz on $V$ with $L=(4a)/r.$
      \end{proof}

\begin{remark}
  The dual of the space $\ell^p,\,0<p<1,$ is the space $\ell^\infty$ of all bounded sequences, the duality $\alpha\mapsto \vphi_\alpha \in\left(\ell^p\right)^*$ for  $\, \alpha=(\alpha_k)\in\ell^\infty,$ being realized by the formula
  $$
  \vphi_\alpha(x)=\sum_{k=1}^\infty\alpha_kx_k,\quad\mbox{for}\;\;x=(x_k)\in\ell^p\,,
  $$
(see \cite[p. 110]{Maddox1}).

  Consequently, for $0<p<1$  every space $\ell^p$ contains a good supply of nonempty open convex sets and non identically null  continuous convex functions.

  In contrast, $\left(L^p[0,1]\right)^*=\{0\}$ for every $0<p<1,$ so that $L^p[0,1]$ does not contain nonempty open convex subsets and the only continuous convex function on $L^p[0,1]$ is $f\equiv 0$ (see \cite[\S1.47]{Rudin-FA}).
  \end{remark}

      A similar result holds in metrizable LCS. Let $\left(X,\tau\right)$ be a Hausdorff LCS with the topology generated by the countable directed family $(p_n)_{n\in\Nat}$ of seminorms. It is known that the topology of $X$ is metrizable and
\bequ\label{eq1.cv-Lip-LCS-m}
      d(x,y)=\sum_{n=1}^\infty\frac1{2^n}\cdot\frac{p_n(x-y)}{1+p_n(x-y)}, \quad x,y\in X\,,
\eequ
is a translation invariant metric on $X$ generating the topology $\tau$.
\begin{prop}\label{cs-p.cv-Lip-LCS-m} Let $X$ be a metrizable  LCS and $\Om$ an open convex subset of $X$. If
$f:\Om\to \Real$ is a continuous convex function, then $f$ is locally Lipschitz on $\Om$ with respect to the metric \eqref{eq1.cv-Lip-LCS-m}
  \end{prop}     \begin{proof}
  Let $x_0\in \Om.$ By Theorem \ref{cs-t.cv-Lip-lcs} there exists a convex neighborhood $U\subset \Om$ of $x_0,\,$ $m\in \Nat$ and $L_m>0$ such that
  \bequ\label{eq2.cv-Lip-LCS-m}
  |f(x)-f(y)|\le L_mp_m(x-y)\,,
  \eequ
  for all $x,y\in U$. Let $r>0$ be such that $V:=\{x\in X : d(x_0,x)\le r\}\subset U\cap \{x\in X : p_m(x-x_0)\le 1\}.$
 Then, for any $x,y\in V,\, p_m(x-y)\le 2$ and
 \begin{align*}
   |f(x)-f(y)|\le& L_mp_m(x-y)=2^mL_m(1+p_m(x-y))\cdot\frac1{2^m}\cdot\frac{p_m(x-y)}{1+p_m(x-y)}\\
   \le& 3\cdot L_m\cdot 2^m\cdot\sum_{k=1}^\infty\frac1{2^k}\cdot\frac{p_k(x-y)}{1+p_k(x-y)}= L\cdot d(x,y)\,,
 \end{align*}
 where $L:=3\cdot L_m\cdot 2^m$
\end{proof}

\begin{remark}
  The fact that the metric $d$ is translation invariant is essential for  the validity of   Propositions \ref{cs-p.cv-Lip-elp} and \ref{cs-p.cv-Lip-LCS-m}.
\end{remark}

Indeed, on  $X=\Real$   the metric $d(x,y)=|x^3-y^3|,\, x,y\in\Real,$ generates the usual topology on $\Real.$
The function $f(x)=x,\, x\in\Real,$ is continuous and convex on $\Real$, but it is not Lipschitz around 0, because
$$
|f(x)-f(y)|=\frac 1{x^2+xy+y^2}\cdot|x^3-y^3|\quad \mbox{for}\;\; (x,y)\ne (0,0)\,,
$$
and
$$
\lim_{(x,y)\to (0,0)}\; \frac 1{x^2+xy+y^2} = +\infty\,.
$$


\begin{thebibliography}{99}
\small{

\bibitem{Alipr-Cones}
C.~D. Aliprantis and R.~Tourky, \emph{Cones and duality}, Graduate Studies in
  Mathematics, vol.~84, American Mathematical Society, Providence, RI, 2007.

\bibitem{Hitchhick}
C.~D. Aliprantis and K.~C. Border, \emph{Infinite-dimensional analysis. A
  hitchhiker's guide}, Studies in Economic Theory, vol.~4, Springer-Verlag,
  Berlin, 1994.

 \bibitem{zali16}
V. Anh Tuan, Ch. Tammer and C. Z\u{a}linescu,
 {\it  The Lipschitzianity of convex vector and set-valued functions}, TOP \textbf{24} (2016), no. 1, 273--299.

\bibitem{bor81}
 J. M. Borwein,
{\it Convex relations in analysis and optimization,}
in Generalized Convexity, Academic Press New York 1981,
pp. 336-377.

\bibitem{bor82}
 \bysame,
{\it Continuity and differentiability properties of convex operators,}
Proc. London Math. Soc.  {\bf 44} (1970), 420-444.

\bibitem{bor84}
\bysame,
{\it Subgradients of convex operators,}
Operationsforsch. Statist. Ser. Optimization {\bf 15} (1984),  179-191.

\bibitem{Breck}
W.~W. Breckner, \emph{Rational {$s$}-convexity. {A} generalized
  {J}ensen-convexity}, Presa Universitar\u a Clujean\u a, Cluj-Napoca, 2011.

\bibitem{b-g-t96}
 W. W. Breckner, A. G\" opfert and T. Trif,
{\it Characterizations of ultrabarrellednes and barrelledness involving
singularities of families of convex mappings,} Manuscripta Math. {\bf 91} (1996), 17-34.

\bibitem{br-tr99}
W. W. Breckner and T. Trif,
{\it Equicontinuity and H\" older equicontinuity of generalized convex
mappings,} New Zealand J. Math. {\bf 28} (1999), 155-170.

\bibitem{vesely13}
A. Carioli, L. Vesel\'y,
{\it Normal cones and continuity of vector-valued convex functions,}
J. Convex Anal. {\bf 20} (2013), no. 2, 495--500.

\bibitem{cob79}
 S. Cobza\c s,
{\it On the Lipschitz properties of convex functions,}
Mathematica {\bf 21} (1979),  123-125.

\bibitem{cobz01a}
 \bysame, {\it Lipschitz properties for families of convex mappings},
  Inequality theory and applications. {V}ol. {I}, Nova Sci. Publ., Huntington,
  NY, 2001, pp.~103--112.

  \bibitem{co-mun76}
 S. Cobza\c s and  I. Muntean,
{\it Continuous and locally Lipschitz convex functions,}
Mathematica {\bf 18} (1976), 41-51.

\bibitem{Ekl-Tem74}
I. Ekeland and R. Temam,
{\it Analyse convexe et probl\` emes variationnels},
Dunod, Paris 1974.

\bibitem{Zali-Gopf}
A.~G{\"o}pfert, H.~Riahi, Ch. Tammer and C.~Z{\u{a}}linescu, \emph{Variational
  methods in partially ordered spaces}, CMS Books in Mathematics/Ouvrages de
  Math\'ematiques de la SMC, 17, Springer-Verlag, New York, 2003.

\bibitem{Hardy-Lit-Pol}
G.~H. Hardy, J.~E. Littlewood and G.~P{\'o}lya, \emph{Inequalities}, Cambridge
  Mathematical Library, Cambridge University Press, Cambridge, 1988, Reprint of
  the 1952 edition.

\bibitem{jou-thi84}
M. Jouak and L. Thibault,
{\it Equicontinuity of families of concave-convex operators,}
Canad. J. Math. {\bf 36} (1984), 883-898.

\bibitem{kos76}
 P. Kosmol,
{\it Optiemierung konvexer Funktionen mit Stabilit\" atsbetrachtungen,}
Dissertationes Math. Vol. 140, pp. 1-42, 1976.

\bibitem{Kos91}
\bysame,
{\it Optimierung und Approximation,} Walter de Gruyter, Berlin-New York 1991.

\bibitem{k-s-w79}
 P. Kosmol, W. Schill and M. Wriedt,
{\it Der Satz von Banach-Steinhaus f\" ur Konvexe Operatore,}
Arch. Math.  33 (1979),  564-569.

\bibitem{KK95}
 A. G. Kusraev and S. S. Kutateladze,
{\it Subdifferentials: Theory and Applications,}
Kluwer A. P. Dordrecht-Boston-London 1995.

\bibitem{Maddox1}
I.~J. Maddox, \emph{Elements of functional analysis}, Cambridge University
  Press, London-New York, 1970, (second edition, 1988).

\bibitem{ma-mu96}
B. Marco and J. A. Murillo,
{\it Locally Lipschitz and convex functions,}
Mathematica {\bf 38} (1996), 121-131.

\bibitem{nem86}
 A. B. N\' emeth,
{\it On the subdifferentiablity of convex operators,}
J. London. Math. Soc. Vol. {\bf 34} (1986),  592-598.

\bibitem{neu85}
 M. Neumann,
{\it Uniform boundedness and closed graph theorems for convex operators,}
Math. Nachr. {\bf 120} (1985),  113-125.

\bibitem{papag83a}
N.~S. Papageorgiou, \emph{Nonsmooth analysis on partially ordered vector   spaces.
{I}. {C}onvex case},   Pacific J. Math. \textbf{107} (1983), no.~2,   403--458.

\bibitem{papag83b}
\bysame, \emph{Nonsmooth analysis on partially ordered vector spaces.
{II}.   {N}onconvex case, {C}larke's theory}, Pacific J. Math. \textbf{109} (1983),   no.~2, 463--495.

  \bibitem{Peressini}
A.~L. Peressini, \emph{Ordered topological vector spaces}, Harper \& Row,
  Publishers, New York-London, 1967.

\bibitem{Bonet}
P.~P{\'e}rez~Carreras and J.~Bonet, \emph{Barrelled locally convex spaces},
  North-Holland Mathematics Studies, vol. 131, North-Holland Publishing Co.,
  Amsterdam, 1987, Notas de Matem{\'a}tica [Mathematical Notes], 113.

  \bibitem{Rudin-FA}
W.~Rudin, \emph{Functional analysis}, second ed., International Series in Pure
  and Applied Mathematics, McGraw-Hill, Inc., New York, 1991.


\bibitem{saxon74}
S.~A. Saxon, \emph{Some normed barrelled spaces which are not {B}aire}, Math.
  Ann. \textbf{209} (1974), 153--160.

\bibitem{Schaef-TVS}
H.~H. Schaefer and M.~P. Wolff, \emph{Topological vector spaces}, second ed.,
  Graduate Texts in Mathematics, vol.~3, Springer-Verlag, New York, 1999.
}

\bibitem{zali78}
C.~Z{\u{a}}linescu, \emph{A generalization of the {F}arkas lemma and
  applications to convex programming}, J. Math. Anal. Appl. \textbf{66} (1978),
  no.~3, 651--678.

  \bibitem{Zali-CvAn}
\bysame, \emph{Convex analysis in general vector spaces}, World Scientific
  Publishing Co., Inc., River Edge, NJ, 2002.
\end{thebibliography}
\end{document}